\documentclass[11pt]{amsart}
\usepackage{oldgerm}
\usepackage{amssymb} % \mathfrak, \mathbb
\usepackage{mathrsfs} % \mathscr
\usepackage{amsmath}
\usepackage{latexsym}
\usepackage[all]{xy}
\usepackage[dvips]{graphics}
\usepackage{amsthm}
\usepackage{enumerate}
\usepackage{bm}
\usepackage{float}
\usepackage{tikz}
\newcommand{\lt}{{}^\mathrm t\hspace{-0.5mm}}
\newcommand{\ctext}[1]{\raise0.2ex\hbox{\textcircled{\scriptsize{#1}}}}
\newcommand{\pmat}[1]{\begin{pmatrix} #1 \end{pmatrix}}

\theoremstyle{plain} %% context in italic
\newtheorem{thm}{Theorem}[section] % Theorem 1.1, 1.2, 1.3, etc.
\newtheorem{prop}[thm]{Proposition} % Proposition 1.1, 1.2, 1.3, etc.
\newtheorem{lem}[thm]{Lemma} % Lemma 1.1, 1.2, 1.3, etc.
\newtheorem{cor}[thm]{Corollary} % Corollary 1.1, 1.2, 1.3, etc.
\theoremstyle{definition} %% context in roman
\newtheorem{defn}[thm]{Definition} % Definition 1.1, 1.2, 1.3, etc.
\newtheorem{rem}[thm]{Remark} % Remark 1.1, 1.2, 1.3, etc.
\theoremstyle{remark} %% context in roman
\makeatletter
\@namedef{subjclassname@2020}{\textup{2020} Mathematics Subject Classification}
\makeatother

\title[Rational curves on a smooth Hermitian surface II]{{\parbox{\textwidth}{\ttfamily\centering\fontsize{13.0pt}{15.0pt}\selectfont Rational curves on a smooth Hermitian surface II:
\\moduli and counting of certain infinite families}}}
\author
{Norifumi Ojiro}
\address{
Deta and Mathematical Sciences Course, 
Graduate School of Software
Information Science, %/ Graduate School of Software and Information Science, 
Iwate Prefectural University \\
152-52, Sugo, Takizawa, Iwate, 020-0693, Japan
}
\email{norifumi.ojiro@gmail.com}

\date{}

\begin{document}
\maketitle

\begin{abstract}
A smooth $k$-Hermitian surface $X$ is a surface projectively isomorphic over $k$ to the Fermat surface of degree $q+1$, where $k$ is an algebraic closure of a finite field with $q^2$ elements.
%Let $X$ be a smooth Hermitian surface in an arbitrary positive characteristic.
%in the projective space $\mathbb{P}^3(k)$ of a positive characteristic $p$, where $q$ is a power of $p$ and $k$.
% $p>0$, where $q$ is a power of $p$. 
%In this paper, 
%we study tetranomial curves on $X$, where tetranomial curves mean the curves which
We consider the curves on $X$ parametrized by polynomials with exactly $4$ terms under a suitable choice of the parameter. We determine the $k$-projective equivalence classes and the moduli spaces of such curves of all degrees. %Further we show that 
It is shown that each equivalence class has only one orbit under the action of the group of the automorphisms of $X$ if $q\ge3$
% acts transitively on  
while if $q=2$ all the classes split into infinitely many orbits, and moreover, the moduli spaces are affine algebraic sets.
%which is
%It is shown that each nonempty class has only one orbit for $q\ge3$ but is divided into infinitely many orbits for $q=2$ under the action of the automorphism group ${\rm Aut}(X)$ of $X$.
%We 
We also determine the number of
%derive formulas counting 
the curves belonging to each orbit.
% on $X$, 
%by determining the subgroup of projective automorphism group of $X$
%${\rm Aut}(X)$ 
%stabilizing a curve of each equivalence class.
% in the orbit.
In addition, the smoothness, the reflexivity and 
the minimal field of definition
%the rationality over finite field
for such curves are shown.
%This work significantly advances the previous result.
% We also remark on 
%investigating the action by the automorphism group of $X$ on each the class of the curves. 
%Investigating the parametric representations the action of the projective automorphism group of $X$ to the curves, 
%Moreover, we count the number of such curves. 
%It is shown that configuration of the curves brings interesting strongly regular graphs and association schemes. 
\end{abstract}

\section{Introduction}
%Let $q$ be a power of a prime $p$ 
%Let $p$ a prime q=p^{\nu}$ for a prime $p$ and $\nu\ge1$ and 
Let $k$ be an algebraic closure of a finite field $\mathbb{F}_{q^2}$ for a prime power $q$.
% with $q=p^{\nu}$ for a prime $p$.
% and $\mathbb{P}^r$ denote the $r$-dimensional $k$-projective space.
%Throughout this paper, a point of the $k$-projective space $\mathbb{P}^3$ is denoted by a column vector: $\bm{x}=\lt(x_0,x_1,x_2,x_3)$. 
%We define 
%$\bm{x}^{(q)}=\lt({x_0}^q,{x_1}^q,{x_2}^q,{x_3}^q)$.
Let $M=(m_{i,j})\mapsto M^{(q)}=({m_{i,j}}^q)$ be the $q$-Frobenius map for an arbitrary matrix $M$ with entries in $k$.
%that is defined by $M^{(q)}=({m_{i,j}}^q)$.
%Then $(MN)^{(q)}=M^{(q)}N^{(q)}$.
We denote a point in a projective space by a column vector and denote the transpose of $M$ by $\lt M$.
%For , we call
%with the entries in $k$, 
%the projective surface
Then a $k$-Hermitian surface is defined by
$$
X_A=\left\{\bm{x}
%=\lt(x_0,x_1,x_2,x_3)
\in\mathbb{P}^3\middle\vert\lt\bm{x}A\bm{x}^{(q)}=0\right\},
$$ 
%defined by $\lt\bm{x}A\bm{x}^{(q)}=0$ for $$ 
where $A$ is a nonzero $4$-by-$4$ matrix with entries in $k$.
% cf.\ \cite{O},\cite{Sh1}. 
It can be easily proven that $X_A$ is smooth if and only if $A$ is invertible.
%Note that $x\mapsto x^q$ is an involutive automorphism on $\mathbb{F}_{q^2}$.
%, which is an analog of the complex conjugate $z\mapsto\bar{z}$ on the complex numbers. 
If $A$ is Hermitian i.e.\ $\lt A=A^{(q)}$ then $X_A$ is simply called a Hermitian surface which is defined over $\mathbb{F}_{q^2}$.
%, and it is defined over $\mathbb{F}_{q^2}$ (cf. Proposition \ref{prop1}). 
Especially, if $A$ is equal to the identity matrix $I$, the surface $X_I$ is the Fermat surface of degree $q+1$.
It is known that all smooth $k$-Hermitian surfaces are projectively isomorphic over $k$ to $X_I$.
Hereafter we assume that the matrix $A$ is invertible i.e. $A\in{\rm GL}_4(k)$.

%A smooth $k$-Hermitian surface $X_A$ has high rationality and symmetry, it
% is special, but
%plays important rolls 
Smooth $k$-Hermitian varieties are not only fundamental objects in algebraic geometry
%in the arithmetic theory of algebraic geometry 
%and are 
but also important for the application to combinatorics, code and lattice, see e.g.\ \cite{BC},\cite{C},\cite{L},\cite{Sh}, because of high rationality and high symmetry.
%across various fields of mathematics.
%foundational
%foundational
%object
%an important object
% object of study 
%in algebraic geometry. 
In fact,
% in positive characteristics, 
it is known in the theory of algebraic surface that
%a smooth $k$-Hermitian surface 
$X_A$ is rational for $q=2$, but non-rational for $q\ge3$ whereas unirational, and so supersingular, especially $K3$ for $q=3$, cf.\ \cite{S},\cite{SK}.
%The surface is also worthwhile 
%In combinatorics,
%Actually, 
%it is known that 
%Also, there exist many results on combinatorics with respect to projective linear varieties
% in $\mathbb{P}^r$
% of a projective space 
%incidental to Hermitian variety, e.g.\ see \cite{BC},\cite{C}.
% the incidence of
%rational points of Hermitian variety and hyperplanes incidental to the variety. 
As for combinatorics on the rational curves of degree $q+1$ on $X_A$, it was recently shown in \cite{O1} that such curves
% on $X_A$ 
produce
%combinatorial structures %such as 
infinite families of certain strongly regular graphs and association schemes, like the case of lines on the surface.
%and interesting 
%algebraic surfaces. 
%since it has fairly high symmetry. %it equips rich combinatorial structure. 
%For example, the configuration of $\mathbb{F}_{q^2}$-rational points and lines on $X_A$ and that of rational normal curves totally tangent to $X_A$ form interesting designs, graphs and so on, see \cite{BC},\cite{C},\cite{Se},\cite{Sh2}. %Thus, studying on rational curves contained in the surface is meaningful.
%Lines on a Hermitian surface, more generally, linear spaces on a $k$-Hermitian hypersurface were actively studied for interest from not only algebraic geometry \cite{Se},\cite{Sh} but also combinatorial designs \cite{BC},\cite{C}. 
%However, there are few results known about rational curves of degree $>1$ on the surface.
%In \cite{KK}, rational curves on a certain K3 surface in characteristic 3 are investigated, where the surface can be regarded as a special case of $k$-Hermitian surfaces. 
%On the other hand, \cite{Sh1} counts the number of rational normal curves totally tangent to a $k$-Hermitian hypersurface. By inspired to the work, 
%In this paper, we study rational curves on $X_A$ which are not lines. %and satisfy a certain condition.

Since our concern is in rational curves of degree $>1$ on $X_A$, it is restricted to non-planar rational curves, that is, rational curves not contained in any plane, because a plane section of $X_A$ is a smooth $k$-Hermitian curve (which is nonrational) or a union of lines.
Let $C_d$ be a rational curve of degree $d$ in $\mathbb{P}^3$. Then there is
%are a unique minimum integer $d>0$ and 
a nonzero 4-by-($d+1$) matrix $F$ with entries in $k$ such that
%$C$ is parametrized by the map
\begin{equation*}
%\mathbb{P}^1\ni{}^\mathrm t\hspace{-0.5mm}(s,t)\longmapsto F\ {}^\mathrm t\hspace{-0.5mm}(s^d,s^{d-1}t,\dots,st^{d-1},t^d)\in C.
C_d=\left\{F\ {}^\mathrm t\hspace{-0.5mm}(s^d,s^{d-1}t,\dots,st^{d-1},t^d)\in\mathbb{P}^3\ \middle\vert\ {}^\mathrm t\hspace{-0.5mm}(s,t)\in\mathbb{P}^1\right\}.
\end{equation*}
%The number $d$ is the degree of $C$. 
%Put $F=\left(\bm{f}_0,\dots,\bm{f}_d\right)$. 
%If there is $1\le i\le d-1$ such that $\left(\bm{f}_0,\bm{f}_i,\bm{f}_d\right)$ has full rank then $C$ has degree $d\ge2$, and moreover,
%If there are $1\le i\neq j\le d-1$
% with $\gcd(d,i,j)=1$ 
%such that $\left(\bm{f}_0,\bm{f}_i,\bm{f}_j,\bm{f}_d\right)$ has full rank then $C$ is non-planar and has $d\ge3$.
Obviously, if $C_d$ is non-planar, the rank of $F$ is $4$ and $d\ge3$.
%Suppose that
%$s,t$ be homogeneous coordinates of $\mathbb{P}^1$ and 
%For $F
%$ be a $4$-by-$4$ invertible matrix, that is, $F
%\in {\rm GL}_4(k)$, define
%Then by a change of parameters $s,t$, suppose that $C$ is expressed as
%$C$ is parameterized by polynomials with $4$ terms. 
%$F$ is reduced to a form whose all columns are zero vectors except for 4 columns.
%Then we consider $C$

Let $F^*\in{\rm GL}_4(k)$ and let $d,i,j$ be positive integers such that $\gcd(d,i,j)=1$ for $1\le i\neq j\le d-1$.
%for $F^*\in{\rm GL}_4(k)$ and integers $d,i,j>0$ such that $1\le i\neq j\le d-1$ and $\gcd(d,i,j)=1$, define
Define
\begin{align*}%\label{cf}
%\bm{v}(d,i,j)&=\lt(s^{d},s^{d-i}t^{i},s^{d-j}t^{j},t^{d}),\\
C_{F^*}(d,i,j)&=\left\{F^*\,\lt(s^{d},s^{d-i}t^{i},s^{d-j}t^{j},t^{d})\in\mathbb{P}^3\ \middle\vert\ {}^\mathrm t\hspace{-0.5mm}(s,t)\in\mathbb{P}^1\right\}.
\end{align*}
%for $F^*\in{\rm GL}_4(k)$ and $\gcd(d,i,j)=1$,
%$1\le i\neq j\le d-1$,
%and
%$$
%%\mathcal{T}_d=\left\{C_F\subset X_A\ \middle\vert\ C_F=C_{F^*}(d,i,j)\ {\rm for}\ {\rm some}\ F^*\in{\rm GL}_4(k)
%%%,\ 1\le i\neq j\le d-1
%%\right\}.
%\mathcal{T}_d=\left\{C_{F^*}(d,i,j)\ \middle\vert\ 
%%{\rm for}\ {\rm some}\ 
%C_{F^*}(d,i,j)\subset X_A
%%F^*\in{\rm GL}_4(k),\ 1\le i\neq j\le d-1
%\right\}.
%$$
%Let $R_d$ be the set of non-planar rational curves of degree $d$ on $X_A$.
%In the previous paper \cite{O}, we studied such curves of degree not more than $q+1$. We showed the following results: If $d<q+1$ then $R_d$ is empty. When $d=q+1$ all elements of $R_{q+1}$ are projectively isomorphic over $k$ to a curve $C(q+1,1,q)$, where
%\begin{equation*}
%C(d,i,j):=\left\{{}^\mathrm t\hspace{-0.5mm}(s^{d},s^{d-i}t^i,s^{d-j}t^j,t^{d})\in\mathbb{P}^3\ |\ {}^\mathrm t\hspace{-0.5mm}(s,t)\in\mathbb{P}^1\right\}.
%\end{equation*}
%Further, the number $|R_{q+1}|$ of $R_{q+1}$ is equal to $q^4(q^3+1)(q^2-1)$.
%Let
%Note that 
Then $C_{F^*}(d,i,j)$ is a non-planar rational curve of degree $d$.
%We consider every curve expressed as $C_{F^*}(d,i,j)$ under a suitable choice of the coordinates $s,t$ of $\mathbb{P}^1$. Thus put
Put
$$
\mathcal{T}_d=\left\{C_d\subset X_A\ \middle\vert\ 
%{\rm for}\ {\rm some}\ 
%C_d\subset X_A,\ 
C_d=C_{F^*}(d,i,j)\ {\rm under\ a\ suitable\ choice\ of}\ s,t
%{\rm for}\ {\rm some}\ F^*,i,j
%F^*\in{\rm GL}_4(k),\ 1\le i\neq j\le d-1
\right\}.
$$
%Note that $F^*$ changes into non-square matrices. 
% if $d$ is the minimal one satisfying the condition. 
%defines a non-planar rational curve for each $F^*,d,i,j$, but these are allowed to take infinitely many possible values
%are not uniquely determined 
%for a curve $C_{F^*}(d,i,j)$ because of 
%the infinite number of the possible choices for the coordinates $s,t$ of $\mathbb{P}^1$.
%especially it is possible that $F^*$ changes into non-square matrices.
% by transforming coordinates of $\mathbb{P}^1$
%and thus we denote by $C_F$ the curve which equals to $C_{F^*}(d,i,j)$ for a choice of the coordinates.
% of $\mathbb{P}^1$.
%We denote by $\mathcal{T}_d$ the set of the curves on $X_A$ which are expressed as \eqref{cf} by changing parameters, if necessary.
%Note that 
%$F,d,i,j$ are not uniquely determined for a curve $C_{F^*}(d,i,j)$, especially it is possible that the coefficient matrix $F$ changes into non-square matrices, and $C_{F^*}(d,i,j)=C_{F^*}(nd,ni,nj)$ for $n$ every positive integer.
% which are projectively isomorphic over $k$ to $C(d,i,j)$ for some $i\neq j$.
%call ``tetranomial curves''. 
%If $\mathcal{T}_d$ is nonempty then all curves in $\mathcal{T}_d$ can be parameterized by polynomials with $4$ terms, and so they are non-planar rational curves.
In the previous paper \cite{O}, we treated the set $\mathcal{R}_d$ of all rational curves of degree $d$ for $2\le d\le q+1$ on $X_A$.
%Let $\mathcal{R}_d$ be 
%The set $R_d$ contains $\mathcal{T}_d$. 
Especially we showed the transitivity of the action of the automorphism group of $X_A$ on $\mathcal{R}_{q+1}$ and gave the number of $\mathcal{R}_{q+1}$.
%the set of the curves of degree $q+1$ 
Here $\mathcal{R}_{q+1}$ is nothing but $\mathcal{T}_{q+1}$.
In this paper, we extend the previous result to the case of $\mathcal{T}_d$ for all $d$, together with adding a lot of new material.
%not mentioned in the previous one.
%give a generalization of the result described as above to the case of $\mathcal{T}_d$. 
The main result is described as follows:
\begin{thm}\label{th0}
If $\mathcal{T}_d$ is nonempty then
%$d$ is equal to
\begin{center}
$d=q+1$, $d=q(q+1)$ for $q$ even, or $d=\tfrac{q(q+1)}{2}$ for $q$ odd,
\end{center}
and then all curves of $\mathcal{T}_d$ 
%all curves in $\mathcal{T}_d$
%$\mathcal{T}_{q+1}$, $\mathcal{T}_{q(q+1)}$ and $\mathcal{T}_{q(q+1)/2}$ 
are projectively isomorphic over $k$ to
%one of the three curves
$$
C_I(q+1, 1, q),\ C_I\left(q(q+1),q+1,q^2+1\right),\ 
%\ {\rm for}\ q\ {\rm even},\\
%\ \begin{array}{c}C_I\left(q(q+1),q+1,q^2+1\right)\\
{\rm or}\ 
C_I\left(\tfrac{q(q+1)}{2},\tfrac{q+1}{2}, \tfrac{q^2+1}{2}\right),
%{\rm for}\ q\ {\rm odd},
$$
%\begin{align*}
%&C_I(q+1, 1, q),\\
%&C_I\left(q(q+1),q+1,q^2+1\right)\ {\rm for}\ q\ {\rm even},\\
%%\ \begin{array}{c}C_I\left(q(q+1),q+1,q^2+1\right)\\
%&{\rm or}\\ 
%&C_I\left(\tfrac{q(q+1)}{2},\tfrac{q+1}{2}, \tfrac{q^2+1}{2}\right)\ {\rm for}\ q\ {\rm odd},
%\end{align*}
respectively.
Further, all curves of $\mathcal{T}_{q+1}$ are smooth but those of $\mathcal{T}_{q(q+1)}$,$\mathcal{T}_{\frac{q(q+1)}{2}}$
%the others 
both have just $\lt(1,0,0,0),\lt(0,0,0,1)$ as singular points, and these curves are all non-reflexive.
\end{thm}
%Put
%$$
%d_1=q+1,\ \ d_2=q(q+1)\ {\rm for}\ 2\mid q,\ \ d_3=\tfrac{q(q+1)}{2}\ {\rm for}\ 2\nmid q.
%%\right\}.
%$$

%Let ${\rm Aut}(X_A)$ be 
The group of the automorphisms of $X_A$ is defined by
\begin{equation*}
{\rm Aut}(X_A)=
\{P\in{\rm GL}_4(k)\ |\ {}^\mathrm t PAP^{(q)}=\lambda A,\ \lambda\in k^{\times}\}/k^{\times}I.
\end{equation*}
Then ${\rm Aut}(X_A)$ naturally acts on $\mathcal{T}_d\neq\emptyset$. The following two theorems completely determine the moduli space of $
%{\rm Aut}(X_A)\backslash
\mathcal{T}_d$.
% i.e. ${\rm Aut}(X_A)\backslash\mathcal{T}_d$.
% \ref{thm1},\ref{modu} are concerned with 
%indicates the number of the ${\rm Aut}(X_A)$-orbits of $\mathcal{T}_d$.
\begin{thm}\label{thm1}
%Let $\mathcal{T}_d$ be the set of tetranomial curves of degree $d$ on a smooth $k$-Hermitian surface $X_A$ of degree $q+1$, and 
%Let ${\rm Aut}(X_A)$ be the projective automorphism group of $X_A$. 
%Suppose that $\mathcal{T}_d$ is nonempty.
%Let $d\in D$.
%satisfied Theorem \ref{th0}. 
%$d\in\left\{q+1,q(q+1),q(q+1)/2\right\}$. 
Let
%$d\in\{d_1,d_2,d_3\}$. 
$$
d\in\left\{q+1,\ \ q(q+1)\ {\rm for}\ 2\mid q,\ \ \tfrac{q(q+1)}{2}\ {\rm for}\ 2\nmid q
\right\}.
$$
Then
%the number of the ${\rm Aut}(X_A)$-orbits of $\mathcal{T}_d$ is
%$\mathcal{T}_d$ are infinitely many ${\rm Aut}(X_A)$-orbits when $q=2$ , on the other hand, for $q\ge3$, ${\rm Aut}(X_A)$ acts transitively on $\mathcal{T}_d$.
$$
|{\rm Aut}(X_A)\backslash \mathcal{T}_d|=
\begin{cases}
\infty&{\rm for}\ q=2,\\
1&{\rm for}\ q\ge3.
\end{cases}
$$
Especially, when $q\ge3$ the moduli space of $
%{\rm Aut}(X_A)\backslash 
\mathcal{T}_{q+1}$ equals to $\left\{(0,0)\right\}\subset\mathbb{A}^2$ and those of $\mathcal{T}_{q(q+1)},\mathcal{T}_{\frac{q(q+1)}{2}}$ equal to $\left\{0\right\}\subset\mathbb{A}^1$.
%and moreover, when $q\ge3$ the stabilizer subgroup of ${\rm Aut}(X_A)$ for a curve of $\mathcal{T}_{q+1}$ is isomorphic to ${\rm PGU}_2(\mathbb{F}_{q^4})$ and those of $\mathcal{T}_{(q+1)q},\mathcal{T}_{(q+1)q/2}$ are isomorphic to commutative groups of order
%$$
%(q+1)(q^3+1),\ \ (q+1)(q^3+1)/2,
%$$
%respectively.
\end{thm}
%We will also give an evaluation for the size of ${\rm Aut}(X_A)\backslash \mathcal{T}_d$ for $q=2$ (Remark \ref{rem1}).
\begin{thm}\label{modu}
Let $q=2$. The set ${\rm Aut}(X_A)\backslash \mathcal{T}_{q(q+1)}$ is bijectively parameterized by
%$1$-dimensional
%toric variety $k/{\mathbb{F}_4}^\times$
$\mathbb{A}^1$,
%over $k$ 
and moreover, ${\rm Aut}(X_A)\backslash \mathcal{T}_{q+1}$ is bijectively parameterized by the disjoint union of $\mathbb{A}^2$ and $A_1\backslash{\rm SL_2}(k)$, where $A_1$ is the cyclic subgroup of order $2$ generated by
{\footnotesize$\pmat{0&1\\1&0}$},
that is the smooth rational hypersurface in $\mathbb{A}^4$ defined by $x^2w+y^2z+xy+1=0$.
% not ${\rm SL_2}(k)$,
% determined by ${\rm SL_2}(k)\rightarrow A_1\backslash{\rm SL_2}(k)$
% which is
%the hypersurface of $\mathbb{A}^4$
%an étale quotient,
%and an algebraic group, 
% anti-diagonal identity matrix of rank $2$.
%, and so
% and ${\rm SL_2}(k)$ defines an affine smooth hypersurface of $\mathbb{A}^4$.
%$$
%\mathbb{A}^2\bigsqcup
%%\left(
%{\rm SL_2}(k)
%%/\langle h\rangle\right)%\ {\rm for}\ h=\pmat{0&1\\1&0}.
%$$
%$\mathbb{A}^4$, more precisely, a certain subset $Z'$ of ${\rm Aut}(X_A)\backslash \mathcal{T}_{q+1}$ is bijectively parameterized by $(\mathbb{A}^2\setminus\{(0,0)\})\times\mathbb{A}^2$ and the complement $U'$ of $Z'$ in ${\rm Aut}(X_A)\backslash \mathcal{T}_{q+1}$ is bijectively parameterized by $\mathbb{A}^4\setminus H\simeq\mathbb{A}^2\setminus\{(0,0)\}$ except for a certain point of $U'$ corresponding to $H$ that is
%, where $H$ is 
%a coordinate plane of $\mathbb{A}^4$.
%$4$-dimensional toric variety.
\end{thm}
As observed above, the connected component of ${\rm Aut}(X_A)\backslash \mathcal{T}_{q+1}$
corresponding to $A_1\backslash{\rm SL_2}(k)$
%$A_1\backslash{\rm SL_2}(k)$
%, which is a smooth hypersurface of $\mathbb{A}^4$, 
is native to $q=2$ and those of ${\rm Aut}(X_A)\backslash \mathcal{T}_{q(q+1)}$, ${\rm Aut}(X_A)\backslash \mathcal{T}_{q+1}$ corresponding to $\mathbb{A}^1,\mathbb{A}^2$ collapse to singleton when $q\ge3$, respectively. As we will see later, the parametrizations are realized by polynomial isomorphism. Hence we have the following.
%By Theorem \ref{modu}, 
\begin{cor}
Let $q=2$. 
%One can equip ${\rm Aut}(X_A)\backslash \mathcal{T}_d$ with the induced topology and structure sheaf. 
The set ${\rm Aut}(X_A)\backslash \mathcal{T}_{q(q+1)}$ is a line isomorhic to
% an algebraic set
%a rational curve 
%affine line 
$\mathbb{A}^1$ and ${\rm Aut}(X_A)\backslash \mathcal{T}_{q+1}$ is an algebraic set isomorphic to
%the algebraic set with
%consisting of 
%two connected components
% isomorphic to 
$\mathbb{A}^2\sqcup A_1\backslash{\rm SL_2}(k)$.
%with respect to naturally induced topology and structure sheaf from each moduli space.
\end{cor}

When $A$ is Hermitian, the surface $X_A$ is defined over $\mathbb{F}_{q^2}$, furthermore if $q=2$ then $X_A$ is rational,   and so all rational curves on $X_A$ are defined over $\mathbb{F}_{q^2}$.
The following result
% concerns the minimal field of definition of curve on $X_A$. It 
guarantees for all $q$ that if $A$ is Hermitian then all curves of $\mathcal{T}_{q+1},\mathcal{T}_{\frac{q(q+1)}{2}}$ can be defined over $\mathbb{F}_{q^2}$ and those of $\mathcal{T}_{q(q+1)}$ can be defined over $\mathbb{F}_{2q^2}\supset\mathbb{F}_{q^2}$, where $\mathbb{F}_{2q^2}=\mathbb{F}_{2^{2\nu+1}}$, that is, the curves we consider admit models over such finite fields.
\begin{cor}\label{cor2}
%Let $d$ be as in Theorem \ref{thm1}.
% be as in Corollary \ref{mainthm}.
Let $C_d$ be an arbitrary curve of $\mathcal{T}_d$ and $C_d(\mathbb{F}_{q^m})$ the set of all $\mathbb{F}_{q^m}$-rational points of $C_d$.
When $A$ is Hermitian,
%Then 
%all curves of $\mathcal{T}_{q+1}$ are 
$C_{q+1}$ is
isomorphic over $\mathbb{F}_{q^2}$ to $\mathbb{P}^1(k)$ and $C_{\frac{q(q+1)}{2}}(\mathbb{F}_{q^2})$ $\left({\rm resp.}\ C_{q(q+1)}(\mathbb{F}_{2q^2})\right)$
%the set of all $\mathbb{F}_{q^2}$,$({\rm resp.}\ \mathbb{F}_{2q^2})$-rational points of each curve of $\mathcal{T}_{\frac{q(q+1)}{2}}$,$({\rm resp.}\ \mathcal{T}_{q(q+1)})$
is birational over $\mathbb{F}_{q^2}$ to $\mathbb{P}^1(\mathbb{F}_{q^2})$ $\left({\rm resp.}\ \mathbb{P}^1(\mathbb{F}_{2q^2})\right)$.
%Then there is at least one ${\rm Aut}(X_A)$-orbit of $\mathcal{T}_{q+1}$ whose all the curves are isomorphic over $\mathbb{F}_{q^2}$ to $\mathbb{P}^1(k)$ and there is at least one ${\rm Aut}(X_A)$-orbit of each $\mathcal{T}_{q(q+1)},\mathcal{T}_{\frac{q(q+1)}{2}}$ in which the set of all $\mathbb{F}_{q^2}$-rational points of each curve is birationally equivalent over $\mathbb{F}_{q^2}$ to $\mathbb{P}^1(\mathbb{F}_{q^2})$.
%Especially, when $q\ge3$ the statement can be valid for all curves in $\mathcal{T}_{q+1},\mathcal{T}_{q(q+1)},\mathcal{T}_{\frac{q(q+1)}{2}}$ instead of orbits.
%all curves of $\mathcal{T}_{q+1}$ are isomorphic over $\mathbb{F}_{q^2}$ to $\mathbb{P}^1(k)$ and those of $\mathcal{T}_{q(q+1)},\mathcal{T}_{\frac{q(q+1)}{2}}$ are birational over $\mathbb{F}_{q^2}$ to $\mathbb{P}^1(\mathbb{F}_{q^2})$.
%of $\mathcal{T}_{3}$
% for $q=2$
%the orbit for each $\mathcal{T}_{q+1},\mathcal{T}_{q(q+1)}$ 
%whose all the curves are defined over $\mathbb{F}_{4}$
%defined over $\mathbb{F}_{q^2}$ for $q\ge3$ and there is at least one ${\rm Aut}(X_A)$-orbit
%$\mathcal{O}_{q+1},\mathcal{O}_{q(q+1)}$ 
%of $\mathcal{T}_{3}$
% for $q=2$
%the orbit for each $\mathcal{T}_{q+1},\mathcal{T}_{q(q+1)}$ 
%whose all the curves are defined over $\mathbb{F}_{4}$. In contrast, any curves in $\mathcal{T}_{q(q+1)},\mathcal{T}_{\frac{q(q+1)}{2}}$ are not defined over $\mathbb{F}_{q^2}$ for $q\ge3$ and there is at least one ${\rm Aut}(X_A)$-orbit
%$\mathcal{O}_{q+1},\mathcal{O}_{q(q+1)}$ 
%of $\mathcal{T}_{6}$ whose any curves are not defined over $\mathbb{F}_{4}$.
\end{cor}
%We explicitly presented such 
%By Corollary \ref{cor2}  
%In the case for $\mathcal{T}_{q+1}$ 
A curve of $\mathcal{T}_{q+1}$ has been explicitly presented in
% the previous paper
\cite{O}.
The following immediately follows from Corollary \ref{cor2}.
\begin{cor}
Let $C_{q+1}$ be as in Corollary \ref{cor2}.
% and $C_{q+1}(\mathbb{F}_{q^2})$ the set of all $\mathbb{F}_{q^2}$-rational points of $C_{q+1}$. 
%Let $S$ be a set of curves and let $S(\mathbb{F}_{q^2})$ denote the set of the restriction of all curves of $S$ to $\mathbb{F}_{q^2}$-rational points.
When $A$ is Hermitian, the automorphism group of $C_{q+1}$ $\left({\rm resp.}\ C_{q+1}(\mathbb{F}_{q^2})\right)$
%every element of $\mathcal{T}_{q+1}(\mathbb{F}_{q^2})=R_{q+1}(\mathbb{F}_{q^2})$
is isomorphic to ${\rm PGL}_2(k)$ $\left({\rm resp.}\ {\rm PGL}_2(\mathbb{F}_{q^2})\right)$.
\end{cor}
On the other hand, the following theorem brings us the number of the curves belonging to an arbitrary ${\rm Aut}(X_A)$-orbit of $\mathcal{T}_d$.
\begin{thm}\label{th2}
%Let $q\ge3$. 
The subgroup of ${\rm Aut}(X_A)$ stabilizing an arbitrary curve of $\mathcal{T}_{q+1}$ is isomorphic to ${\rm PGU}_2(\mathbb{F}_{q^4})$ and those of $\mathcal{T}_{q(q+1)},\mathcal{T}_{\frac{q(q+1)}{2}}$ are isomorphic to the commutative groups
%commutative groups of order
\begin{align*}
%\mathbb{Z}/(q+1)\mathbb{Z}\times\mathbb{Z}/(q^3+1)\mathbb{Z}
%\mathbb{Z}_{q+1}\times\mathbb{Z}_{q^3+1},%\mathbb{Z}_{(q+1)/2}\times\mathbb{Z}_{q^3+1},
&\mathbb{Z}/(q+1)\mathbb{Z}\times\mathbb{Z}/(q^3+1)\mathbb{Z}\ {\rm for}\ 2\mid q,\\
%\mathbb{Z}/\left(\tfrac{q+1}{2}\right)\mathbb{Z}
&2\mathbb{Z}/\left(q+1\right)\mathbb{Z}\times\mathbb{Z}/(q^3+1)\mathbb{Z}\ {\rm for}\ 2\nmid q,
\end{align*}
respectively. 
%where we denote $\mathbb{Z}_m=\mathbb{Z}/m\mathbb{Z}$. %for the ring of rational integers $\mathbb{Z}$.
\end{thm}
%It should be noted that 
%Theorem \ref{th2} implies that the stabilizer group of a curve of $\mathcal{T}_d$ is independent of the choice of the curve.
%Hence we have the following corollary.
%, and so the number $|\mathcal{O}_d|$ presented in Corollary \ref{mainthm} is constant for every curve in $\mathcal{T}_d$.
\begin{cor}\label{mainthm}
Let $d$ be as in Theorem \ref{thm1} and $\mathcal{O}_d$ an arbitrary ${\rm Aut}(X_A)$-orbit of $\mathcal{T}_d$.
% for $d\in\{d_1,d_2,d_3\}$.
%Suppose that $\mathcal{T}_d$ is nonempty. 
%If $d\in D$ 
Then
%for $d=q+1,q(q+1)$. Then we have $|\mathcal{O}_d|=|\mathcal{T}_d|$.
%Let $q$ be as in Theorem \ref{th2}. Then we have
%Let $\mathcal{T}_d$ be as in Theorem \ref{thm1}. 
%If $\mathcal{T}_d\neq\emptyset$, it is one of the following three cases:
%\begin{itemize}
%\item[C.1)]$d=q+1$, all elements of $\mathcal{T}_d$ are projectively isomorphic over $k$ to the smooth curve $C(q+1, 1, q)$, moreover
\begin{align*}
|\mathcal{O}_{q+1}|&=q^4(q^3+1)(q^2-1),%\ {\rm for}\ q\ge3,
%\begin{cases}
%\infty&{\rm for}\ q=2\\q^4(q^3+1)(q^2-1)&{\rm for}\ q\ge3,\\
%\end{cases}
\\
%\end{equation*}
%
%\item[C.2)]$d=(q+1)q$ with even $q$, all elements of $\mathcal{T}_d$ are projectively isomorphic over $k$ to the singular curve $C((q+1)q,q+1,q^2+1)$, moreover
%\begin{equation*}
|\mathcal{O}_{q(q+1)}|&=q^6(q^4-1)(q-1)\ {\rm for}\ 2\mid q,%{\rm even},%\ \ge4,
%\begin{cases}
%\infty&{\rm for}\ q=2\\q^6(q^4-1)(q-1)&{\rm for}\ {\rm even}\ q\ge4,\\
%\end{cases}
%\end{equation*}
\\
%\item[C.3)]$d=(q+1)q/2$ with odd $q$, all elements of $\mathcal{T}_d$ are projectively isomorphic over $k$ to the singular curve $C((q+1)q/2, (q+1)/2, (q^2+1)/2)$, moreover
%\begin{equation*}
\left|\mathcal{O}_{\frac{q(q-1)}{2}}\right|&=2q^6(q^4-1)(q-1)\ {\rm for}\ 2\nmid q,
%\ {\rm odd},
%\end{equation*}
\end{align*}
especially if $q\ge3$ then $|\mathcal{T}_d|=|\mathcal{O}_d|$
% take the above numbers 
by Theorem \ref{thm1}.
%and moreover, if $q\ge3$ then $|\mathcal{T}_d|=|\mathcal{O}_d|$ for all $d\in D$.
%\end{itemize}
%For $q=2$, let $\mathcal{O}_d$ be an ${\rm Aut}(X_A)$-orbit of $\mathcal{T}_d$ for $d=q+1,q(q+1)$. Then we have $|\mathcal{O}_d|=|\mathcal{T}_d|$.
\end{cor}
In the next section we prove Theorem \ref{th0}, preparing a lemma and combining it with a well-known theorem.
%We also remark on the smoothness and the reflexivity of the curves of $\mathcal{T}_d$.
% for $d=q+1,q(q+1)$.
%, we give a necessary condition for which %$\mathcal{T}_d\neq\emptyset$, and determine $\mathcal{T}_d$ up to projective isomorphism over $k$. 
In Section $3$, after
%preparing some notations 
some preparation we prove Theorem \ref{thm1} and Theorem \ref{modu}
% evaluate the size of ${\rm Aut}(X_A)\backslash \mathcal{T}_d$ for $q=2$ 
by
%after introducing some notation, 
%we associate 
%showing that there is an one-to-one correspondence between 
relating ${\rm Aut}(X_A)\backslash \mathcal{T}_d$ to the set of certain matrices.
In the last one shows that Corollary \ref{cor2} is immediately obtained from Theorem \ref{th0} and Theorem \ref{thm1}.
 %expressing the curves of $\mathcal{T}_d$, 
%and counting the number of the set.
%defined for each $d\in D$ that satisfy certain condition. 
%Using the result we evaluate the size of ${\rm Aut}(X_A)\backslash \mathcal{T}_d$ for $q=2$.
%and show that there is a one-to-one correspondence between them. 
%Moreover we deduce that the number of the ${\rm Aut}(X_A)$-orbits of $\mathcal{T}_d$ is exactly $1$ if $q\ge3$, and otherwise it is an infinite number.
Section $4$ is devoted to the proof of Theorem \ref{th2}, and then we have Corollary \ref{mainthm}.
%, and so Corollary \ref{mainthm}, 
%finally we have Corollary \ref{cor2} by proving a lemma.
%the $\mathbb{F}_{q^2}$-rationality for the curves in the orbits by using the above results.
%In Section $4$, 
%computing the stabilizer of an element of $\mathcal{T}_d$ for $q\ge3$, we complete the proof of Corollary \ref{mainthm}.
%The last Section 5 shows three constructions of strongly regular graph and association scheme attached to $\mathcal{T}_{q+1}$, together with numerical examples.

\section{
%Proof of Theorem \ref{th0}
%Determining 
Determination of projective equivalence classes}
%where $C_I(d,i,j)=C(d,i,j)$. 
%For the curve $C(d,i,j)$, we may assume $i<j$ since $C(d,i,j)$ and $C(d,j,i)$ are projectively equivalent by $P=\small{\pmat{1&0&0&0\\0&0&1&0\\0&1&0&0\\0&0&0&1}}$.
%% under the changes of the parameter: $(s,t)\mapsto(s^{1/n},t^{1/n})$.
%%Then we will prove the following lemma:
%Under this assumption, we will prove the following:
Since $\gcd(d,i,j)=1$ by the definition of $C_{F^*}(d,i,j)$,
%$C_{F^*}(nd,ni,nj)$ for a nonzero $n\in\mathbb{Z}$ is identical
% as a curve 
%to $C_{F^*}(d,i,j)$,
%since a curve is independent of the change of the parameter $(s,t)\mapsto(s^a,t^a)$. 
%Therefore 
we may identify $(nd,ni,nj)$ with $(d,i,j)$ for a nonzero $n\in\mathbb{Z}$.
%Then the following statement holds.
\begin{lem}\label{mat}
Let $A,F^*\in{\rm GL}_4(k)$.
%Identify $C_{F^*}(d,i,j)$ with its the ${\rm Aut}(X_A)$-orbit.
Up to $k$-projective equivalence, a curve $C_{F^*}(d,i,j)$ is contained in $X_A$
%contained in $X_A$ 
if and only if $(d,i,j)$ and $\lt {F^*}A{F^*}^{(q)}$ are of
%is contained in $X_A$ 
the following three cases under the above identification:
%$\left(d,i,j,\lt {F^*}A{F^*}^{(q)}\right)$
%is identically zero for all $s,t$. 
%Let
%$$
%f_B(d,i,j)=(s^{d},s^{d-i}t^{i},s^{d-j}t^{j},t^{d})\,B\,\lt(s^{d},s^{d-i}t^{i},s^{d-j}t^{j},t^{d})^{(q)},
%$$
%%$s,t$ be homogeneous coordinates of $\mathbb{P}^1$ and 
%where $B$ is a $4$-by-$4$ invertible matrix with the entries in $k$. 
%Suppose $B\in{\rm GL}_4(k)$. 
%Assume that $i<j$.
%Then $f_B(d,i,j)$ is identically zero for $s,t$ if and only if $(d,i,j,B)$ 
%is either of the two cases:
\begin{itemize}
\item[1.] $(d,i,j)=(q+1,1,q)$,
\begin{equation*}
\lt {F^*}A{F^*}^{(q)}=
\begin{pmatrix}
0&b_{11}&b_{12}&b_{13}\\
0&b_{21}&b_{22}&b_{23}\\
-b_{11}&-b_{12}&-b_{13}&0\\
-b_{21}&-b_{22}&-b_{23}&0\\
\end{pmatrix}, 
\end{equation*}
where the submatrix formed by rows $1$-$2$ and colums $2$-$4$ is of rank $2$,
%the first two rows are $k$-linearly independent and
%\begin{align*}
%(b_{11},b_{21}),(b_{13},b_{23})\neq(0,0),\ and\ if\ q\ge3\ then\ 
%(b_{12},b_{22})=(0,0).
%\end{align*}
and $(b_{12},b_{22})=(0,0)$ if $q\ge3$.
\\
\item[2.] $(d,i,j)=\left((q+1)q,q+1,q^2+1\right)$ for $2\mid q$,
\begin{equation*}
\lt {F^*}A{F^*}^{(q)}=
\begin{pmatrix}
0&b_{1}&0&b_{3}\\
0&0&0&b_{2}\\
0&0&b_{2}&0\\
b_{1}&b_{3}&0&0\\
\end{pmatrix},
\end{equation*}
where $b_1b_2\neq0$, and
$b_3=0$ if $q\ge3$,
\\
\item[3.] $(d,i,j)=\left((q+1)q/2,(q+1)/2,(q^2+1)/2\right)$ for $2\nmid q$, 
%the matrix $B$ is of the form in {\rm II}.
\begin{equation*}
\lt {F^*}A{F^*}^{(q)}=
\begin{pmatrix}
0&b_{1}&0&0\\
0&0&0&b_{2}\\
0&0&-b_{2}&0\\
-b_{1}&0&0&0\\
\end{pmatrix},
\end{equation*}
where $b_1b_2\neq0$.
\end{itemize}
\end{lem}
\begin{proof}
%Further, 
%note that $(d,d-j,d-i,JBJ)$ can be identified with $(d,i,j,B)$, where $J$ is the $4$-by-$4$ backward identity matrix, since $f_{JBJ}(d,d-j,d-i)=$ is projectively isomorphic to $C(d,i,j)$ by $J$.
%Similarly, 
%$$
Put 
\begin{center}
$B=\left(B_{l,m}\right)=\lt {F^*}A{F^*}^{(q)},\ \ f_B(d,i,j)=\bm{v}\,B\,\lt\bm{v}^{(q)}$,
\end{center}
where $\lt\bm{v}=\lt\left(s^{d},s^{d-i}t^{i},s^{d-j}t^{j},t^{d}\right)\in\mathbb{P}^3$, and
$$
E=\left(E_{l,m}\right)=
\begin{pmatrix}
0&qi&qj&qd\\
i&(q+1)i&i+qj&i+qd\\
j&qi+j&(q+1)j&j+qd\\
d&qi+d&qj+d&(q+1)d
\end{pmatrix},
$$
where the $(l,m)$-component of $E$ is the exponent of $t$ in the monomial of $f_B(d,i,j)$ whose coefficient
%Denoting the entries of $B$ by $B_{lm}$ for $1\le l,m\le 4$,
% and those of $E$ by $E_{lm}$, 
%Denoting $B=(B_{l,m})$, 
is the $(l,m)$-component of $B$.
One may assume $i<j$ since $C_{F^*}(d,i,j)$ and $C_{F^*}(d,j,i)$ are $k$-projective equivalent.
%, since the another case $j<i$ 
%is obtained by replacing $(i,j,B)$ to $(j,i,PBP)$, where $P=\small{\pmat{1&0&0&0\\0&0&1&0\\0&1&0&0\\0&0&0&1}}$.
Suppose that $C_{F^*}(d,i,j)\subset X_A$.
% is the coefficient of the monomial containing $t^{E_{l,m}}$. 
Then $f_B(d,i,j)$ is identically zero, and so we have some linear equations with respect to $d,i,j$ consisting of the components of $E$ since $B$ is non-zero. Since it is easy to see that the equations have the form $E_{l,m}=E_{l',m'}=E_{i'',m''}$ is impossible from the definition of $E$ and $B$, we have some equations of the form
$
E_{l,m}^{l',m'}:E_{l,m}=E_{l',m'}
$
%with respect to $d,i,j$, 
and then $B_{l,m}+B_{l',m'}=0$.
Other components of $B$ not corresponding to such $E_{l,m}^{l',m'}$ are zero, in particular, $B_{1,1},B_{2,1},B_{3,4},B_{4,4}=0$.
%, because $E_{1,1},E_{2,1},E_{3,4},E_{4,4}$ are not equal to any other components of $E$.
To determine $(d,i,j,B)$ it is sufficient to choose two independent equations of $E_{l,m}^{l',m'}$, because if we choose three of such equations then $(d,i,j)$ has to be $(0,0,0)$.
%but this contradicts to the definition.
%Hence if one chooses three $\mathbb{Q}$-linearly independent equations $E_{lml'm'}$'s then $(d,i,j)=(0,0,0)$, this is impossible. On the other hand, if one chooses two such equations then $d,j$ can be written by using only $i,q$.
%In fact, we can choose two equations as follows. 
Note that $E_{3,1}$ can be equal to only $E_{1,2}$ or $E_{2,2}$, similarly $E_{2,4}$ can be equal to only $E_{3,3}$ or $E_{4,3}$ by the definition of $E$ and the assumption $i<j$.
%First of all, note that one can exclude $E_{1,1},E_{2,1},E_{3,4},E_{4,4}$, since these are not equal to any other entries of $E$, and furthermore $E_{31}$ and $E_{24}$ can be equal to only $E_{12}$ or $E_{22}$ and $E_{33}$ or $E_{43}$, respectively. 
%As $B$ to be rank $4$, 
Since ${\rm rank}(B)=4$ we have two equations $E_{3,1}^{1,2}$,$E_{2,4}^{4,3}$, and then $d=(q+1)i$, $j=qi$, thus $(d,i,j)$ is of the case 1 under the identification. Further, these equations precisely lead to $E_{4,1}^{2,2}$, $E_{1,4}^{3,3}$, and if $q=2$ then $E_{3,2}^{1,3}$, $E_{4,2}^{2,3}$, and thus $B$ is as in the case 1.
%For every positive integer $n$ 
%Note that $(i(q+1),i,iq)$ may be identified with $(q+1,1,q)$, since they define the same curve.
%$C_{F^*}(d,i,j)=C_{F^*}(nd,ni,nj)$ for $n\in\mathbb{Z}$ with $n>0$.
%, since $f_B(dn,in,jn)$ is equal to $f_B(d,i,j)$ by the change of the parameter: $(s,t)\mapsto(s^{1/n},t^{1/n})$.
% under $((q+1)i,i,qi,B)=(q+1,1,q,B)$. 
Other possible choices of the two independent equations
%$E_{l,m}^{l',m'}$'s 
are $\left\{E_{4,1}^{1,2},E_{3,3}^{2,4}\right\}$ and $\left\{E_{3,1}^{2,2},E_{4,3}^{1,4}\right\}$.
%Since these lead to
%%Without loss of generality, 
%$C_{F^*}(d,i,j)$ and $C_{F^*}(d,d-j,d-i)$ which are $k$-projective equivalent, we may only consider the former.
%because the latter equals to the former by $(d,d-j,d-i,JBJ)=(d,i,j,B)$, and 
Firstly, from the former one has $d=qi$, $(q+1)j=(q^2+1)i$, and if $q=2$ then $E_{4,2}^{1,4}$, and so $B$ is of the cases 2,3.
%To determine $(d,i,j)$, we note that
%\begin{equation*}
%\gcd(q^2+1,q+1)=
%\begin{cases}
%1&{\rm for}\ q\ {\rm even}\\
%2&{\rm for}\ q\ {\rm odd}.
%\end{cases} 
%\end{equation*}
%To determine $(d,i,j)$, we note that
Since
\begin{equation*}
\gcd(q^2+1,q+1)=
\begin{cases}
1&{\rm for}\ q\ {\rm even}\\
2&{\rm for}\ q\ {\rm odd},
\end{cases} 
\end{equation*}
%Then it immediately follows
%from the two equations in $d,i,j$ 
%that
%let $n$ be a positive integer, 
there is a nonzero $n\in\mathbb{Z}$ such that
\begin{center}
$i=n(q+1)$, $j=n(q^2+1)$ and $d=nq(q+1)$ for $q$ even, 
\end{center}
\begin{center}
$i=n(q+1)/2$, $j=n(q^2+1)/2$ and $d=nq(q+1)/2$ for $q$ odd.
\end{center}
%Theorefe, by
%\begin{center}
%$((q+1)qn,(q^2+1)n,(q+1)n,B)=((q+1)q,q^2+1,q+1,B)$,
%\end{center}
%\begin{center}
%$((q+1)qn/2,(q^2+1)n/2,(q+1)n/2,B)=((q+1)q/2,(q^2+1)/2,(q+1)/2,B)$,
%\end{center}
%they are equal to the cases II, III.
Under the identification, these equal to the cases 2,3, respectively.
%ying the multiplication by $n$.
%This implies the case 2 by the identification, and 
%, respectively, identifying the multiplication by $n$.
%Since the opposite assertion can be easily verified by direct computation, 
Similarly for the latter $\left\{E_{3,1}^{2,2},E_{4,3}^{1,4}\right\}$, one has
\begin{center}
$d=q(q+1)$, $j=q^2-1$ and $i=q-1$ for $q$ even,
\end{center} 
\begin{center}
$d=q(q+1)/2$, $j=(q^2-1)/2$ and $i=(q-1)/2$ for $q$ odd.
\end{center}
Therefore the latter can be identified to the former because $C_{F^*}(d,i,j)$ and $C_{F^*}(d,d-j,d-i)$ are $k$-projective equivalent.
This completes the proof since the opposite assertion is obvious.
%Since one can easily verify that every possible choice of $E_{l,m}^{l',m'}$'s, up to the identification, yields the case I, II or III, the lemma has been proven.
\end{proof}
To show the existence of curves $C_{F^*}(d,i,j)$ with $\lt {F^*}A{F^*}^{(q)}$ of Lemma \ref{mat}, we apply the following important theorem, see \cite{PM},\cite[Proposition 2.5.]
{Sh} for a proof.
\begin{prop}\label{prop1}
For every $A\in{\rm GL}_n(k)$, there is an element $B\in{\rm GL}_n(k)$ such that $A={}^\mathrm t\hspace{-0.5mm} BB^{(q)}$, where if $A$ is Hermitian then $B\in{\rm GL}_n(\mathbb{F}_{q^2})$.
\end{prop}
%Put
%$$
%d_1=q+1,\ \ \ d_2=q(q+1).%\ {\rm for}\ 2\mid q,\ \ \  d_3=q(q+1)/2\ {\rm for}\ 2\nmid q.
%$$
Put
$$
d_1=q+1,\ \ d_2=q(q+1)\ {\rm for}\ 2\mid q,\ \ d_3=\tfrac{q(q+1)}{2}\ {\rm for}\ 2\nmid q.
%\right\}.
$$
Let $\mathcal{B}_{d_l}$
%,\mathcal{B}_{d_2},\mathcal{B}_{d_3}$ 
be the set of the matrices $\lt {F^*}A{F^*}^{(q)}$
%appearing on 
%the right-hand side of the equality 
in the case $l=1,2,3$ of Lemma \ref{mat}. 
%respectively.
%with $l=1,2,3$
% that is of RHS of the equality of each the case of Lemma \ref{mat}
%With the aid of 
By Proposition \ref{prop1} we have ${F_l}^*\in{\rm GL}_4(k)$ such that $B_l=\lt {{F_l}^*}A{{F_l}^*}^{(q)}$ for a given $A\in{\rm GL}_4(k)$ and $B_l\in\mathcal{B}_{d_l}$.
%for every $A,B\in{\rm GL}_4(k)$, there is $F\in{\rm GL}_4(k)$ such that $\lt FAF^{(q)}=B$. Obviously, $C_{F^*}(d,i,j)$ is contained in $X_A$ if and only if $C(d,i,j)$ is contained in $X_B$ if and only if $f_B(d,i,j)$ is identically zero for all $s,t$. 
Then the curve $C_{{F_l}^*}(d_l,i,j)$ for $(d_l,i,j)$ in the case $l$ satisfies the condition of Lemma \ref{mat}, and thus $\mathcal{T}_{d_l}$ is nonempty.

%The statement for smoothness and reflexivity of curves of $\mathcal{T}_{d_l}$ follows from a result of 
According to {\cite[\S4]{RV}},
% is contained in $X_A$. 
%This completes the proof of Theorem \ref{th0}.
% by Lemma \ref{mat}. 
%if $\mathcal{T}_d\neq\emptyset$ then all curves of $\mathcal{T}_d$ are isomorphic to one of the three curves 
%the classes of projective isomorphism of $\mathcal{T}_d$ 
%described in Theorem \ref{th0}.
%projectively isomorphic over $k$ to one of the three curves of Corollary \ref{mainthm}.
%\begin{rem}[cf.\ \cite{RV}]\label{rem2}
%All curves in $\mathcal{T}_{q+1}$ are smooth, because
%the curves are projectively isomorphic over $k$ to the smooth curve $C_I(q+1,1,q)$.
% which is defined by
%$$
%x_1x_3^{q-1}-x_2^q=0,\ \ \ x_0x_3-x_1x_2=0.
%$$
%The defining equations for monomial space curves in $\mathbb{P}^3$ have been determined by \cite{RV}.
the curve $C_I(q+1,1,q)$ is obtained as the set-theoretic complete intersection of the two surfaces defined by
$$
{x_0}^{q-1}x_2-x_1^{q}=0,\ \ \ x_0x_3-x_1x_2=0
$$
which is known to be smooth, non-planar and non-reflexive,
%curve of degree $q+1$
see \cite{BH}.
Similarly,
%all curves of $\mathcal{T}_{q(q+1)}$, $\mathcal{T}_{q(q+1)/2}$ are singular, because 
%every $C'\in 
%all curves in $\mathcal{T}_{q(q+1)}$ are projectively isomorphic over $k$ to 
$C_{d_2}=C_I(q(q+1),q+1,q^2+1)$ is the s.t.c.i. of the two surfaces defined by
%whose the defining ideal is generated by
% whose affine part at $x_0=1$ is 
%the intersection of two surfaces 
%defined by
$$
x_0^{q-1}x_3-{x_1}^q=0,\ \ \ x_1{x_3}^q-{x_2}^{q+1}=0.
$$
%(cf.\ \cite[\S4]{RV}).
%cf.\ \cite[pp.\ 337-338]{RV}.
%whose the Jacobian matrix is of rank $1$ at $\lt(0,0,0)$.
It can be verified by a routine calculation that the set of the singular points of $C_{d_2}$ is 
%$(x_0,x_1,x_2,x_3)=(1,0,0,0)$ and $(0,0,0,1)$,
$$
%{\rm Sing}(C_{d_2})=
\left\{(1,0,0,0),(0,0,0,1)\right\},
$$
and the dual surface
%${C_{d_2}}^\vee$ 
is defined by 
$$
x_0^{q-1}x_2^{q+1}+x_1^{q+1}x_3^{q-1}=0,
$$
and moreover, the projection from the conormal variety of $C_{d_2}$ to the dual surface
%${C_{d_2}}^\vee$ 
has inseparable degree $q$,
% or $2q$ depending on the parity of $q$
and so $C_{d_2}$ is non-reflexive.
Further, $C_{d_3}=C_I\left(\tfrac{q(q+1)}{2},\tfrac{q+1}{2}, \tfrac{q^2+1}{2}\right)$ is the s.t.c.i. of 
the two surfaces 
%defined by
$$
x_0^{q-1}x_3-{x_1}^q=0,\ \ \ {x_0}^{\tfrac{q-1}{2}}{x_2}^{\tfrac{q+1}{2}}-{x_1}^{\tfrac{q+1}{2}}{x_3}^{\tfrac{q-1}{2}}=0,
$$
and it is easy to verify that the set of the singular points is the same with that of $C_{d_2}$. The dual surface is defined by
$$
{x_0}^{\tfrac{q-1}{2}}{x_2}^{\tfrac{q+1}{2}}+{x_1}^{\tfrac{q+1}{2}}{x_3}^{\tfrac{q-1}{2}}=0,
$$
and $C_{d_3}$ is non-reflexive since the projection
%${\rm Con}(C_{d_3})\rightarrow{C_{d_3}}^\vee$ 
from the conormal variety of $C_{d_3}$ to the dual surface
%${C_{d_3}}^\vee$
has inseparable degree $q$.
%Hence $C_I(q(q+1),q+1,q^2+1)$ is singular, thus so is $C'$. 
%By the same argument, it follows immediately that $C_I(q(q+1)/2,(q+1)/2,(q^2+1)/2)$ is singular at $\lt(0,0,0)$, and hence all curves of $\mathcal{T}_{q(q+1)/2}$ are singular.
%\end{rem}
This completes the proof of Theorem \ref{th0}.
%\section{Proof of Theorems \ref{thm1} and \ref{modu}}
\section{
%Proof of Theorem \ref{th0}
Determination of moduli spaces}
%We denote 
Let $\mathcal{M}_{m,n}(R)$ be the set of $m$-by-$n$ matrices with the entries in a ring $R$, it is abbreviated by $\mathcal{M}_{m,n}$ if $R=k$ and $\mathcal{M}_m(R)$ if $m=n$.
By definition, $\mathcal{B}_{d_l}\subset\mathcal{M}_4$ for $l=1,2,3$.
%For each $l=1,2,3$, 
We define
% us define the sets $\mathcal{Q}_{d_1}',\mathcal{Q}_{d_2}',\mathcal{Q}_{d_3}'$ as follows:
\begin{align*}
\mathcal{Q}_{d_l}'&=\left\{(M_{i,j})\in\mathcal{M}_{d_l+1}\ \middle\vert\ M_{i,j}=\begin{cases}B_{f_l(i),f_l(j)}&{\rm if}\ (i,j)\in {D_{d_l}}^2\\0&{\rm if}\ (i,j)\not\in {D_{d_l}}^2
\end{cases},\ B\in\mathcal{B}_{d_l}
\right\},
%\mathcal{Q}_{d_2}'&=\left\{M\in\mathcal{M}_{q(q+1)+1}\ \middle\vert\ M_{i,j}=\begin{cases}B_{f_2(i),f_2(j)}&{\rm for}\ i,j\in D_{d_2}\\0&{\rm for}\ i,j\not\in D_{d_2}
%\end{cases}\ {\rm and}\ B\in\mathcal{B}_{d_2}
%\right\},\\
%\mathcal{Q}_{d_3}'&=\left\{M\in\mathcal{M}_{q(q+1)/2+1}\ \middle\vert\ M_{i,j}=\begin{cases}B_{f_3(i),f_3(j)}&{\rm for}\ i,j\in D_{d_3}\\0&{\rm for}\ i,j\not\in D_{d_3}
%\end{cases}\ {\rm and}\ B\in\mathcal{B}_{d_3}
%\right\},
\end{align*}
where $f_l:D_{d_l}\rightarrow\{1,2,3,4\}$ are the order-preserving bijections and
\begin{align*}
D_{d_1}&=\left\{1,2,q+1,d_1+1\right\},\\
D_{d_2}&=\left\{1,q+2,q^2+2,d_2+1\ \middle\vert\ {\rm for}\ q\ {\rm even}\right\},\\
D_{d_3}&=\left\{1,\tfrac{q+1}{2}+1,\tfrac{q^2+1}{2}+1,d_3+1\ \middle\vert\ {\rm for}\ q\ {\rm odd}\right\}.
\end{align*} 
% defined by
%\begin{align*}
%&f_1(1)=1,\ &f_1&(2)=2,\ &f_1&(q+1)=3,\ &f_1(d_1+1)=4,\\
%&f_2(1)=1,\ &f_2&(q+2)=2,\ &f_2&(q^2+2)=3,\ &f_2(d_2+1)=4,\\
%&f_3(1)=1,\ &f_3&((q+1)/2+1)=2,\ &f_3&((q^2+1)/2+1)=3,\ &f_3(d_3+1)=4.
%\end{align*}
%Let $\mathcal{B}_{d_1}$, $\mathcal{B}_{d_2}$, $\mathcal{B}_{d_3}$ be the sets of the matrices $B$ in cases I, II, III of Lemma \ref{mat}, respectively. 
%We denote by ${\mathcal{Q}_{d_1}}'$ the set
%%$$
%%\mathcal{B}_{d_1}\ni B\mapsto B'\in\mathcal{M}_{q+2};\ \ B'_{i,j}=
%%\begin{cases}
%%B_{i,j}&i,j\in \{1,2\},\\
%%B_{i-(q-2),j}&i\in \{q+1,q+2\},\ j\in \{1,2\},\\
%%B_{i,j-(q-2)}&i\in \{1,2\},\ j\in \{q+1,q+2\},\\
%%B_{i-(q-2),j-(q-2)}&i,j\in \{q+1,q+2\}.\\
%%\end{cases}
%%$$
%consisting of matrices in $\mathcal{M}_{q+2}$ whose entries are zero except the $(i,j)$-entries with 
%$$
%i,j\in D_{d_1}:=\{1,2,q+1,q+2\}
%$$ 
%which form a matrix belonging to $\mathcal{B}_{d_1}$ by the map preserving order: $D_{d_1}\rightarrow\{1,2,3,4\}$. 
%Similarly for even $q$, we denote by ${\mathcal{Q}_{d_2}}'$ the set consisting of matrices in $\mathcal{M}_{(q+1)q+1}$
%%((q+1)q+1)$-by-$((q+1)q+1)$ matrices 
%whose entries are zero except the $(i,j)$-entries with
%$$
%i,j\in D_{d_2}:=\{1,q+2,q^2+2,(q+1)q+1\}
%$$
%which form a matrix belonging to $\mathcal{B}_{d_2}$, and for odd $q$, denote by ${\mathcal{Q}_{d_3}}'$ the set consisting of matrices in $\mathcal{M}_{(q+1)q/2+1}$ whose entries are zero except the $(i,j)$-entries with
%$$
%i,j\in D_{d_3}:=\{1,(q+1)/2+1,(q^2+1)/2+1,(q+1)q/2+1\}
%$$
%which form a matrix belonging to $\mathcal{B}_{d_3}$.

Let ${\rm Im}(\varphi_{d_l})$ be the image of the group homomorphism
$$
\varphi_{d_l}:{\rm GL}_2(k)\ni g\mapsto\varphi_{d_l}(g)\in{\rm GL}_{{d_l}+1}(k)
$$
defined by
\begin{equation*}
\lt(u^{d_l},u^{d_l-1}v,\dots,v^{d_l})=\varphi_{d_l}(g)\,\lt(s^{d_l},s^{d_l-1}t,\dots,t^{d_l})\ {\rm for} \ \lt(u,v)=g\,\lt(s,t).
\end{equation*}
Then ${\rm Im}(\varphi_{d_l})$ is a subgroup of ${\rm GL}_{d_l+1}(k)$.
%For each $l=1,2,3$, 
Further, we define
\begin{align*}
\mathcal{Q}_{d_l}&=\left\{M\in\mathcal{M}_{d_l+1}\ \left|\ \lt\varphi_{d_l}(g_l)M\varphi_{d_l}(g_l)^{(q)}\in {\mathcal{Q}_{d_l}}'\ {\rm for}\ {\rm some}\ g_l\in{\rm GL}_2(k)\right\}\right.,\\
\mathcal{S}_{d_l}&=\left\{F\in\mathcal{M}_{4,d_l+1}\ \left|\ \lt FAF^{(q)}\in\mathcal{Q}_{d_l}\right\}\right..
\end{align*}
%Note that $F=(\bm{f}_1,\cdots,\bm{f}_{d_i+1})$ with $\bm{f}_j=\bm{0}$ for $j\not\in D_i$ if $\lt FAF^{(q)}\in \mathcal{Q}_i'$. 
%\tc{red}{上以降の$F$は字体を変えたほうが良い}
%Since 
The group ${\rm Im}(\varphi_{d_l})$
%${\mathcal{Q}_i}'\subset \mathcal{Q}_i$, and 
%${\rm Im}(\varphi_{d_i})$ 
acts on $\mathcal{S}_{d_l}$ as $F\mapsto F\varphi_{d_l}(g)$ for all $g\in{\rm GL}_2(k)$.
Indeed, since there is $g_l\in{\rm GL}_2(k)$ for each $F\in \mathcal{S}_{d_l}$ by definition such that
$$
\lt\varphi_{d_l}(g_l)\,\lt FAF^{(q)}\,\varphi_{d_l}(g_l)^{(q)}\in{\mathcal{Q}_{d_l}}',
$$
we have
$$
\lt\varphi_{d_l}(g^{-1}g_l)\,\lt\left\{F\varphi_{d_l}(g)\right\}A\left\{F\varphi_{d_l}(g)\right\}^{(q)}\,\varphi_{d_l}(g^{-1}g_l)^{(q)}\in{\mathcal{Q}_{d_l}}',
$$
and so $F\varphi_{d_l}(g)\in \mathcal{S}_{d_l}$.
\begin{defn}
Let $*$ be the map
\begin{align*}
&\mathcal{M}_{d_l+1}\rightarrow\mathcal{M}_{4};\ M=(M_{i,j})\mapsto M^*=(M^*_{i,j}),\\
&\mathcal{M}_{4,d_l+1}\rightarrow\mathcal{M}_{4};\ N=(\bm{n}_1,\dots,\bm{n}_{d_l+1})\mapsto N^*=({\bm{n}^*}_{1},{\bm{n}^*}_{2},{\bm{n}^*}_{3},{\bm{n}^*}_{4}),\\
&\mathcal{M}_{d_l+1,4}\rightarrow\mathcal{M}_{4};\ G=(\bm{g}_1,\bm{g}_2,\bm{g}_3,\bm{g}_{4})\mapsto \lt\left\{\left(\lt G\right)^*\right\},
%=(\lt{\bm{f}^*}_{1},\lt{\bm{f}^*}_{2},\lt{\bm{f}^*}_{3},\lt{\bm{f}^*}_{4}), 
%\ \ ({\rm resp.}\ \varphi_{d_i}(g)\mapsto\varphi_{d_i}(g)^*),
\end{align*}
% $M^*$
% (resp. $\varphi_{d_i}(g)^*$) 
%denotes the $4$-by-$4$ matrix having $(D_{i,l},D_{i,m})$-entry of $M$
%% (resp. $\varphi_{d_i}(g)$) 
%as $(l,m)$-entry with $1\le l,m\le4$.
%Similarly
% for $F\in S_i$, 
%we consider
where $M^*_{i,j}=M_{{f_l}^{-1}(i),{f_l}^{-1}(j)}$ and ${\bm{n}^*}_{j}=\bm{n}_{{f_l}^{-1}(j)}$. 
\end{defn}
Note that $d_l\ge3$, in particular, $d_l=3$ if and only if  $l=1$ and $q=2$, then the map is identical.
%, so we generalize the definition of the map so that it is identical on $\mathcal{M}_{4}$.
% $d_l=3$ i.e. 
Also, it is obvious that
%note that the map 
$*$ is commutative with the transpose and the Frobenius map, i.e.\ $\left(\lt M\right)^{*}=\lt\left(M^{*}\right)$ and $\left(M^{(q)}\right)^{*}=\left(M^{*}\right)^{(q)}$ for all $M$ in the domain of $*$.
%and so
\begin{rem}\label{RM3}
By definition, for every $F\in\mathcal{S}_{d_l}$ there is $g_l\in{\rm GL}_2(k)$ such that $\lt\varphi_{d_l}(g_l)\lt FAF^{(q)}\varphi_{d_l}(g_l)^{(q)}\in {\mathcal{Q}_{d_l}}'$. Since ${\mathcal{Q}_{d_l}}'\rightarrow{{\mathcal{Q}_{d_l}}'}^*=\mathcal{B}_{d_l}\subset{\rm GL}_4(k)$ is bijective and
$$
\left(\lt\varphi_{d_l}(g_l)\lt FAF^{(q)}\varphi_{d_l}(g_l)^{(q)}\right)^*=\lt\left[\left\{F\varphi_{d_l}(g_l)\right\}^*\right]A\left[\left\{F\varphi_{d_l}(g_l)\right\}^*\right]^{(q)},
$$
we have $\left\{F\varphi_{d_l}(g_l)\right\}^*\in{\rm GL}_4(k)$. 
%As shown above, 
%Namely, by choosing a suitable $g_l\in{\rm GL}_2(k)$,
%%Such $g_l$ only depends on $F$ and $d_l$,
%%it is always guaranteed that
%\begin{center}
%$\lt\left\{F\varphi_{d_l}(g_l)\right\}A\left\{F\varphi_{d_l}(g_l)\right\}^{(q)}\in{\mathcal{Q}_{d_l}}'$ and $\left\{F\varphi_{d_l}(g_l)\right\}^*\in{\rm GL}_4(k)$
%\end{center}
%for all $F\in \mathcal{S}_{d_l}$.
\end{rem}
%From the remark, it  that 
Thus, the following proposition is obtained, also cf.\  \cite[Section 3]{O}.
%Then there is the following bijection. 
\begin{prop}\label{prop31}
For each $l=1,2,3$, the map
\begin{equation*}
\begin{array}{ccc}
k^{\times}\backslash \mathcal{S}_{d_l}/{\rm Im}(\varphi_{d_l})&\longrightarrow& \mathcal{T}_{d_l}\\
\rotatebox{90}{$\in$}&&\rotatebox{90}{$\in$}\\
k^{\times}F{\rm Im}(\varphi_{d_l})&\longmapsto&
C_{F^*}(d_l,i_l,j_l)
%C_{\left\{F\varphi_{d_l}(g_l)\right\}^*}(d_l,i_l,j_l).
\end{array}
\end{equation*}
%\begin{equation*}
%k^{\times}\backslash \mathcal{S}_{d_l}/{\rm Im}(\varphi_{d_l})\ni k^{\times}F{\rm Im}(\varphi_{d_l})\longmapsto 
%C_{\left\{F\varphi_{d_l}(g_l)\right\}^*}(d_l,i_l,j_l)
%%C_F
%\in \mathcal{T}_{d_l}
%%\ \ {\rm for}\ l=1,2,3,%\label{bi1}
%\end{equation*}
is a bijection, and therefore
%where $F$ is taken so that $\lt FAF^{(q)}\in \mathcal{Q}_{d_l}'$ and $k^{\times}=k\setminus\{0\}$ 
%Moreover, 
%the map
$$
{\rm Aut}(X_A)\backslash \mathcal{S}_{d_l}/{{\rm Im}(\varphi_{d_l})}\rightarrow{\rm Aut}(X_A)\backslash \mathcal{T}_{d_l}
$$
is also a bijection, where ${\rm Aut}(X_A)$ acts on $k^{\times}\backslash \mathcal{S}_{d_l}/{{\rm Im}(\varphi_{d_l})}$ and $\mathcal{T}_{d_l}$ by the matrix multiplication from the left.
%ting $d_1=q+1$, $d_2=(q+1)q$, $d_3=(q+1)q/2$, we abbreviate $\varphi_{d_1}$, $\varphi_{d_2}$, $\varphi_{d_3}$ to $\varphi_{d_1}$, $\varphi_{d_2}$, $\varphi_3$, respectively.
%For $i=1,2,3,$ let
%\begin{align*}
%\mathcal{Q}_i&=\left\{M\in\mathcal{M}_{d_i+1}\ \left|\ \lt\varphi_{d_i}(g)M\varphi_{d_i}(g)^{(q)}\in {\mathcal{Q}_i}'\ {\rm for}\ {\rm some}\ g\in{\rm GL}_2(k)\right\}\right.,\\
%S_i&=\left\{F\in\mathcal{M}_{4,d_i+1}\ \left|\ \lt FAF^{(q)}\in \mathcal{Q}_i\right\}\right..
%\end{align*}
\end{prop}
Note that
%in the above proposition
%the well-definedness of 
%the first map is well-defined that
%$k^{\times}\backslash \mathcal{S}_{d_l}/{\rm Im}(\varphi_{d_l})\rightarrow \mathcal{T}_{d_l}$ 
%is obvious from the definition of
% the curve 
$C_{F^*}(d_l,i_l,j_l)=C_{\left\{F\varphi_{d_l}(g)\right\}^*}(d_l,i_l,j_l)$ for all $g\in{\rm GL}_2(k)$ in the above proposition.
%$$
%C_{\left\{F\varphi_{d_l}(g_l)\right\}^*}(d_l,i_l,j_l)
%=C_{\left\{F\varphi_{d_l}(h_l)\right\}^*}(d_l,i_l,j_l)
%$$
%for every $g_l,h_l\in{\rm GL}_2(k)$ such that $\left\{F\varphi_{d_l}(g_l)\right\}^*,\left\{F\varphi_{d_l}(h_l)\right\}^*\in{\rm GL}_4(k)$.
%the curve $C_{\left\{F\varphi_{d_l}(g_l)\right\}^*}(d_l,i_l,j_l)$ is independent on the choice of $g_l$.
% and an element of $k^{\times}F{\rm Im}(\varphi_{d_l})$.
%a single $g_l\in{\rm GL}_2(k)$, that is, .
%$$
%C_{\left\{F\varphi_{d_l}(g_l)\right\}^*}(d_l,i_l,j_l)
%=C_{\left\{F\varphi_{d_l}(h_l)\right\}^*}(d_l,i_l,j_l)
%$$
%for every $g_l,h_l\in{\rm GL}_2(k)$.
 
Further, ${\rm Im}(\varphi_{d_l})$ acts on $\mathcal{Q}_{d_l}$ as $M\mapsto\lt\varphi_{d_l}(g)M\varphi_{d_l}(g)^{(q)}$ for all $g\in{\rm GL}_2(k)$.
%and
%${\mathcal{Q}_i}'\subset \mathcal{Q}_i$, and 
%${\rm Im}(\varphi_{d_i})$ 
%acts on $S_i$ as $F\mapsto F\varphi_{d_i}(g)$. 
We denote by $\sim$ the equivalence relation induced by this action, and by $M^{\varphi_{d_l}}$ the equivalence class of $M\in \mathcal{Q}_{d_l}$. 
%Then by identifying a rational curve with its parametric representations (cf. \cite[Section 3]{O}), the set $\mathcal{T}_{d_i}$ can be regarded as $k^{\times}\backslash S_i/{{\rm Im}(\varphi_{d_i})}$, where $k^{\times}=k\setminus\{0\}$.
%Let $D_{i,j}$ denote the $j$-th element of the set $D_i$ for $i=1,2,3$.
%Then
%for $M\in\mathcal{M}_{d_i+1}$, (resp. $\varphi_{d_i}(g)\in{\rm Im}(\varphi_{d_i})$), 
%we 
%Then, 
% as the $j$-th column of $F^*$.
%$F^*$ denotes the $4$-by-$4$ matrix having $D_{i,j}$-th column of $F$ as $j$-th column with $1\le j\le4$.
%Then if $\lt FAF^{(q)}$ for $F\in S_i$ belongs to ${\mathcal{Q}_i}'$, 
%Note that $$
\begin{lem}\label{bico}
For each $l=1,2,3$, the map
\begin{equation*}
\begin{array}{ccc}
{\rm Aut}(X_A)\backslash \mathcal{S}_{d_l}/{{\rm Im}(\varphi_{d_l})}&\longrightarrow& k^{\times}\backslash \mathcal{Q}_{d_l}/\sim\\
\rotatebox{90}{$\in$}&&\rotatebox{90}{$\in$}\\
{\rm Aut}(X_A)F\,{{\rm Im}(\varphi_{d_l})}&\longmapsto&k^{\times}({}^\mathrm t\hspace{-0.5mm}{F}A{F}^{(q)})^{\varphi_{d_l}}
\end{array}
\end{equation*}
is a bijection.
\end{lem}
\begin{proof}
%Well-definedness follows immediately from the definition.
% that the map is well-defined.\\
{\bf Surjectivity}: For every $k^{\times}M^{\varphi_{d_l}}\in k^{\times}\backslash \mathcal{Q}_{d_l}/\sim$,
% to $M$
%, 
one may assume $M\in {\mathcal{Q}_{d_l}}'$ by choice of a suitable $g\in{\rm GL}_2(k)$, and then $M^*\in\mathcal{B}_{d_l}$. Since $A,M^*\in{\rm GL}_4(k)$, there are $G,H\in{\rm GL}_4(k)$ such that $M^*=\lt GG^{(q)}$ and $A=\lt HH^{(q)}$ by Proposition \ref{prop1}. Hence by putting $F^*=H^{-1}G$, one has $\lt F^*A{F^*}^{(q)}=M^*$. Now, put $F=(\bm{f}_1,\cdots,\bm{f}_{d_l+1})\in\mathcal{M}_{4,d_l+1}$
% the $4$-by-$(d_i+1)$ matrix having 
so that 
$$
\bm{f}_j=\begin{cases}
\bm{f}^*_{f_l(j)}&{\rm for}\ j\in D_{d_l}\\
\bm{0}&{\rm for}\ j\not\in D_{d_l},
\end{cases}
$$
where $F^*=(\bm{f}^*_1,\bm{f}^*_2,\bm{f}^*_3,\bm{f}^*_4)$.
%columns of $F$ are zero vectors except the $D_{i,j}$-th  columns which equal to $j$-th columns of $F^*$.
%whose $j$-th column vector is one of $F^*$  zero column vectors 
%Appending zero column vectors to $$-th columns of $F^*$ appropriately, we obtain $F\in S_i$ such that $\lt FAF^{(q)}=M$.
Then $\lt FAF^{(q)}=M$, this proves the surjectivity.
% has been proven.  
%Adding zero column vectors to $F^*$ appropriately, we obtain $F\in S_i$ such that $\lt FAF^{(q)}=M$.\\

{\bf Injectivity}: Let $F,G\in \mathcal{S}_{d_l}$ such that $k^{\times}(\lt FA{F}^{(q)})^{\varphi_{d_l}}=k^{\times}(\lt GA{G}^{(q)})^{\varphi_{d_l}}$. 
%We may assume that ${\lt\varphi_{d_i}(g)}\lt {F}A{{F}}^{(q)}{\varphi_{d_i}(g)}^{(q)},\lt GAG^{(q)}\in \mathcal{Q}_i'$,
% assume that 
%if necessary, by selecting appropriate $g$.
%act on $F,G$.
%$\lt FA{F}^{(q)}, \lt GA{G}^{(q)}\in {\mathcal{Q}_i}'$, 
%Noting that the entries of $F,G$ are zero except the columns corresponding to the $4$ columns of $F^*$, one has
Take $g_l\in{\rm GL}_2(k)$ so that ${G}{\varphi_{d_l}(g_l)}$ satisfies Remark \ref{RM3}.
%such that $\left({F}{\varphi_{d_l}(g_1)}\right)^*$ and $\left({G}{\varphi_{d_l}(g_2)}\right)^*$ belong to ${\rm GL}_4(k)$
%replace $G$ by it.
%${F}{\varphi_{d_l}(g_1)},{G}{\varphi_{d_l}(g_2)}$, respectively.
Then there are $\lambda\in k^\times$ and $h_l\in{\rm GL}_2(k)$
such that
\begin{equation}\label{eql3.3.1}
\lambda\,\lt\left\{F\varphi_{d_l}(h_l)\right\}A\left\{{F}{\varphi_{d_l}(h_l)}\right\}^{(q)}=\lt\left\{{G}{\varphi_{d_l}(g_l)}\right\}A\left\{{G}{\varphi_{d_l}(g_l)}\right\}^{(q)},
\end{equation}
and then by
%the columns of $F\varphi_{d_i}(g),G$ are zero vectors except the $D_{i,j}$-th columns for $1\le j\le4$, and so
\begin{align*}
\left[\lt\left\{F\varphi_{d_l}(h_l)\right\}A\left\{{F}{\varphi_{d_l}(h_l)}\right\}^{(q)}\right]^*&=\left[\lt\left\{F\varphi_{d_l}(h_l)\right\}\right]^*A\left[\left\{{F}{\varphi_{d_l}(h_l)}\right\}^{(q)}\right]^*
,\\ 
\left[\lt\left\{{G}{\varphi_{d_l}(g_l)}\right\}A\left\{{G}{\varphi_{d_l}(g_l)}\right\}^{(q)}\right]^*&=\left[\lt\left\{{G}{\varphi_{d_l}(g_l)}\right\}\right]^*A\left[\left\{{G}{\varphi_{d_l}(g_l)}\right\}^{(q)}\right]^*,
\end{align*}
%Hence by the definition of $F,G$,
we have
\begin{equation*}
\lambda\,\left[\lt\left\{F\varphi_{d_l}(h_l)\right\}\right]^*A\left[\left\{{F}{\varphi_{d_l}(h_l)}\right\}^{(q)}\right]^*=\left[\lt\left\{{G}{\varphi_{d_l}(g_l)}\right\}\right]^*A\left[\left\{{G}{\varphi_{d_l}(g_l)}\right\}^{(q)}\right]^*.
%\ \ {\rm for}\ {\rm some}\ \lambda\in k^{\times}.
\end{equation*}
Since $\left\{{G}{\varphi_{d_l}(g_l)}\right\}^*\in{\rm GL}_4(k)$, so is $\left\{{F}{\varphi_{d_l}(h_l)}\right\}^*$ by the above equality, and so
$$
k^{\times}\left\{{F}{\varphi_{d_l}(h_l)}\right\}^*{\left\{{G}{\varphi_{d_l}(g_l)}\right\}^*}^{-1}
%,k^\times G^*{\varphi_{d_i}(g)^*}^{-1}{{F}^*}^{-1}
\in{\rm Aut}(X_A),
$$
%Therefore we conclude that
%or equivalently,
%$
%k^{\times}\left({F}{\varphi_{d_i}(g)}\right)^*
%\in{\rm Aut}(X_A){{G}^*}$
%{\rm Aut}(X_A)k^{\times}F\,{{\rm Im}(\varphi_{d_i})}={\rm Aut}(X_A)k^{\times}G\,{{\rm Im}(\varphi_{d_i})}. 
or equivalently
\begin{equation}\label{eql3.3.2}
k^{\times}\left\{{F}{\varphi_{d_l}(h_l)}\right\}^*
%,k^\times G^*{\varphi_{d_i}(g)^*}^{-1}{{F}^*}^{-1}
\in{\rm Aut}(X_A){\left\{{G}{\varphi_{d_l}(g_l)}\right\}^*}.
\end{equation}
Since both sides of the equality \eqref{eql3.3.1} belong to ${\mathcal{Q}_{d_l}}'$, putting ${\left\{{G}{\varphi_{d_l}(g_l)}\right\}^*}=(\bm{g}^*_1,\bm{g}^*_2,\bm{g}^*_3,\bm{g}^*_{4})$ (resp. ${\left\{{F}{\varphi_{d_l}(h_l)}\right\}^*}=(\bm{f}^*_1,\bm{f}^*_2,\bm{f}^*_3,\bm{f}^*_{4})$),
%$\lt (G\varphi_{d_l}(g_1))A(G\varphi_{d_l}(g_1))^{(q)}\in{\mathcal{Q}_{d_l}}'$, 
%for ${\left\{{G}{\varphi_{d_l}(g_1)}\right\}^*}=(\bm{g}^*_1,\bm{g}^*_2,\bm{g}^*_3,\bm{g}^*_{4})$, 
the matrix ${G\varphi_{d_l}(g_l)}=(\bm{g}_1,\dots,\bm{g}_{d_l+1})$ (resp. ${F\varphi_{d_l}(h_l)}=(\bm{f}_1,\dots,\bm{f}_{d_l+1})$) is of %the form
$$
\bm{g}_j\ ({\rm resp.}\ \bm{f}_j)=\begin{cases}
\bm{g}^*_{f_l(j)}\ ({\rm resp.}\ \bm{f}^*_{f_l(j)})&{\rm for}\ j\in D_{d_l}\\
\bm{0}&{\rm for}\ j\not\in D_{d_l}.
\end{cases}
$$
%and so by the equality \eqref{eql3.3.1}, the matrix $F\varphi_{d_l}(g)$ is also of the same form.
%\begin{center}
%%the $(i,j)$-entry of
%%${G}{\varphi_{d_l}(g_2)}$ is
%$
%{G\varphi_{d_l}(g_2)}_{i,j}=
%\begin{cases}
%{\left\{{G}{\varphi_{d_l}(g_2)}\right\}^*}_{{f_l}(i),{f_l}(j)}&(i,j)\in{D_{d_l}}^2\\
%0&(i,j)\not\in{D_{d_l}}^2.
%\end{cases}
%$
%\end{center}
Hence by \eqref{eql3.3.2},
$$
k^{\times}{F}{\varphi_{d_l}(h_l)}
%,k^\times G^*{\varphi_{d_i}(g)^*}^{-1}{{F}^*}^{-1}
\in{\rm Aut}(X_A){G}{\varphi_{d_l}(g_l)},
$$
and thus
$$
{\rm Aut}(X_A)F\,{{\rm Im}(\varphi_{d_l})}={\rm Aut}(X_A)G\,{{\rm Im}(\varphi_{d_l})}.
$$
This proves the injectivity.
\end{proof}

\begin{rem}
It immediately follows from Lemma \ref{bico} that
the map
\begin{equation*}
\begin{array}{ccc}
k^{\times}\backslash \mathcal{S}_{d_l}/{{\rm Im}(\varphi_{d_l})}&\longrightarrow& k^{\times}\backslash \mathcal{Q}_{d_l}/\sim\\
\rotatebox{90}{$\in$}&&\rotatebox{90}{$\in$}\\
k^\times F\,{{\rm Im}(\varphi_{d_l})}&\longmapsto&k^{\times}({}^\mathrm t\hspace{-0.5mm}{F}A{F}^{(q)})^{\varphi_{d_l}}
\end{array}
\end{equation*}
is surjective but not injective for $l=1,2,3$.
\end{rem}

\begin{lem}\label{lem3}
For all $l=1,2,3$ we have
\begin{equation*}
|k^{\times}\backslash \mathcal{Q}_{d_l}/\sim|=
\begin{cases}
\infty&\ {\rm for}
%\ i=1,2\ {\rm and}
\ q=2\\
1&\ {\rm for}
%\ i=1,2,3,\ {\rm and}
\ q\ge3.
\end{cases}
\end{equation*}
\end{lem}
\begin{rem}\label{remRl}
Let 
%$R_l$ be the element of ${\mathcal{Q}_{d_l}}'$ for each $l=1,2,3$
%$R_i\in{\mathcal{Q}_i}'$ 
%such that
\setlength{\arraycolsep}{3pt}
\begin{equation*}
{R_1}^*=
\begin{pmatrix}
0&0&0&-1\\
0&1&0&0\\
0&0&1&0\\
-1&0&0&0\\
\end{pmatrix},
%\begin{pmatrix}
%0&1&0&0\\
%0&0&0&1\\
%-1&0&0&0\\
%0&0&-1&0\\
%\end{pmatrix},
\ 
{R_2}^*=
\begin{pmatrix}
0&1&0&0\\
0&0&0&1\\
0&0&1&0\\
1&0&0&0\\
\end{pmatrix},
\ 
{R_3}^*=
\begin{pmatrix}
0&1&0&0\\
0&0&0&1\\
0&0&-1&0\\
-1&0&0&0\\
\end{pmatrix}.
\end{equation*}
Then $R_l\in{\mathcal{Q}_{d_l}}'$, and if $q\ge3$ then $R_l$ forms a complete system of representatives of
%the quotient set 
$k^{\times}\backslash \mathcal{Q}_{d_l}/\sim$ by Lemma \ref{lem3}.
\end{rem}
\begin{proof}
%Throughout the proof, suppose that $M_l=\left(M_{i,j}^l\right)\in{{\mathcal{Q}_{d_l}}'}$ for $k^{\times}M_l^{\varphi_{d_l}}\in k^{\times}\backslash \mathcal{Q}_{d_l}/\sim$ by choosing a suitable $g_l\in{\rm GL}_2(k)$,
%%acting an appropriate $\varphi_{d_i}(g)$ on $M^i$ 
%and so ${M_l}^*=\left({M^l_{i,j}}^*\right)\in\mathcal{B}_{d_l}$.
% is a matrix of the form described in Lemma \ref{mat}.
%Put ${M^i}^*=({M_{j,k}^i}^*)$. 
%We refer to the entries of $M^*$ as $a_{ij}(M^*)$ or $b_i(M^*)$ by following the description of the matrices $B$ of Lemma \ref{mat}.\\
\mbox{}\\
{\bf (Case where $q\ge3$)}
\\
%we have nothing to prove since it 
For all $k^{\times}M_l^{\varphi_{d_l}}\in k^{\times}\backslash \mathcal{Q}_{d_l}/\sim$, one may assume that
%by choosing a suitable $g_l\in{\rm GL}_2(k)$, 
$M_l=\left(M_{i,j}^l\right)\in{{\mathcal{Q}_{d_l}}'}$
%acting an appropriate $\varphi_{d_i}(g)$ on $M^i$ 
and then ${M_l}^*=\left({M^l_{i,j}}^*\right)\in\mathcal{B}_{d_l}$ by changing the representative.
Since the case of $\mathcal{Q}_{d_1}$ has been already proven in \cite{O}, we will show those of $\mathcal{Q}_{d_2},\mathcal{Q}_{d_3}$.
% of $B$ in Lemma \ref{mat}, 
In the cases $l=2,3$, when $q\ge3$, by definition ${M^l_{1,4}}^*,{M^l_{4,2}}^*=0$ for all $k^{\times}M_l^{\varphi_{d_l}}\in k^{\times}\backslash \mathcal{Q}_{d_l}/\sim$.
%\begin{center}
%${M_{1,3}^1}^*,{M_{2,3}^1}^*,{M_{3,2}^1}^*,{M_{4,2}^1}^*=0$, and 
%\end{center}
Then there is a matrix
$
%\begin{equation*}
g_l=
\begin{pmatrix}
\lambda_l&0\\
0&\mu_l\\
\end{pmatrix}
\in {\rm GL}_2(k)
%\end{equation*}
$
such that 
$$
\left\{\lt\varphi_{d_l}(g_l)\right\}^*{M_l}^*\left\{\varphi_{d_l}(g_l)^{(q)}\right\}^{*}={R_l}^*,
%\ \ {\rm for}\ i=2,3.
$$
where $R_l$ is of Remark \ref{remRl}.
Indeed, the above equation in $g_l$ is equivalent to
% the two equations
%the simultaneous equations in $\lambda,\mu$ :
\begin{equation*}
%\begin{cases}
{M^l_{1,2}}^*\lambda_l^{d_lq}\mu_l^{d_l}=1,\ \ \ \ 
{M^l_{2,4}}^*\lambda_l^{(d_l-r_l)(q+1)}\mu_l^{r_l(q+1)}=1,
%\end{cases}
\end{equation*}
where 
%$b_1={M^l_{1,2}}^*=-{M^l_{4,1}}^*$, $b_2={M^l_{2,4}}^*=-{M^l_{3,3}}^*$ and
$$
r_l=
\begin{cases}
q^2+1&{\rm for}\ l=2,\\
(q^2+1)/2&{\rm for}\ l=3.
\end{cases}
$$
Since $k$ is algebraically closed, 
%for each $l$ 
the solutions
%$(\lambda_l,\mu_l)$ of the 
% definitely 
exist in $k$. Hence the assertion for this case has been proven.
\\
{\bf (Case where $q=2$)}\\
For the case, we separate the proof for $\mathcal{Q}_{d_1}$ and that for $\mathcal{Q}_{d_2}$.\\
{\bf Case of $\mathcal{Q}_{d_1}$:}
%For $k^{\times}\backslash \mathcal{Q}_{d_1}/\sim$, 
%we consider the following matrices: 
%putting
Let
\begin{equation*}
%k\ni x\mapsto
{M}^*(x)=
\begin{pmatrix}
0&x&1&0\\
0&1&0&1\\
x&1&0&0\\
1&0&1&0
\end{pmatrix}
\in\mathcal{B}_{d_1}
\ \ 
{\rm for}\ x\in k,
%\ \ {\rm with}\ x\in k.
\end{equation*}
%Then the matrices are of the case I of Lemma \ref{mat}.
and let $M(x)=\left(M(x)_{i,j}\right)$ the matrix of ${\mathcal{Q}_{d_1}}'$ defined by 
$$
M(x)_{i,j}=
\begin{cases}
{M}^*(x)_{f_1(i),f_1(j)}&{\rm for}\ i,j\in D_{d_1},\\
0&{\rm for}\ i,j\not\in D_{d_1}.
\end{cases}
$$
%and form $M(x)\in \mathcal{Q}_{d_1}'$ so that their entries are zero except the $(D_{1,l},D_{1,m})$-entries which equal to $(l,m)$-entries of $M^*(x)$.
Assume that $k^{\times}M(x)^{\varphi_{d_1}}=k^{\times}M(x')^{\varphi_{d_1}}$ for $x,x'\in k$. Then there are $\lambda\in k^\times$ and $g\in{\rm GL}_2(k)$ such that
$\lambda\,\lt\varphi_{d_1}(g)M(x)\varphi_{d_1}(g)^{(q)}=M(x')$,
%Since $M(x)\in{\mathcal{Q}_{d_1}}'$ the above equality implies
and so
$$
\lambda\left\{\lt\varphi_{d_1}(g)\right\}^*{M}^*(x)\left\{\varphi_{d_1}(g)^{(q)}\right\}^*={M}^*(x').
$$
Then by direst calculation,
it turn out that
% if 
%$$
%\lt\left[\left\{\varphi_{d_1}(g)\right\}^*\right]M^*(x)\left[\left\{\varphi_{d_1}(g)\right\}^*\right]^{(q)}=M^*(x')
%$$
%then 
the matrix $g$ must be scalar, and then $x=x'$.
%Therefore if $k^{\times}M^*(x)^{\varphi_{d_1}^*}=k^{\times}M^*(x')^{\varphi_{d_1}^*}$ then $x=x'$, and so
%Therefore
%for $k^{\times}M(x)^{\varphi_{d_1}},k^{\times}M(x')^{\varphi_{d_1}}\in k^{\times}\backslash \mathcal{Q}_{d_1}/\sim$  with $x\neq x'$,  one has 
Therefore
%we conclude that 
if $x\neq x'$ then $k^{\times}M(x)^{\varphi_{d_1}}\neq k^{\times}M(x')^{\varphi_{d_1}}$. 
%belong to distinct classes of $k^{\times}\backslash \mathcal{Q}_{d_1}/\sim$. 
This implies the existence of an injection from $k$ to $k^{\times}\backslash \mathcal{Q}_{d_1}/\sim$, and thus
%the map $x\mapsto k^{\times}M(x)^{\varphi_{d_1}}$ is injective.
%Since
% there is an infinite number of the choices of $x$
%$|k|=\infty$, we have 
$|k^{\times}\backslash \mathcal{Q}_{d_1}/\sim|=\infty$.
\\
{\bf Case of $\mathcal{Q}_{d_2}$:}
%For $k^{\times}\backslash \mathcal{Q}_{d_2}/\sim$,
%every $L\in {\mathcal{Q}_{d_2}}'$ is of
%\begin{center}
%$L^*_{1,2}=L^*_{4,1}=1,\ \ L^*_{2,4}=L^*_{3,3}=1$.
%\end{center}
Firstly, note that every $M=\left(M_{i,j}\right)\in{\mathcal{Q}_{d_2}}'$
%with $2\mid q$ 
can be transformed into the form with
\begin{equation}\label{q2f}
M^*_{1,2}=M^*_{4,1}=1,\ \ M^*_{2,4}=M^*_{3,3}=1
\end{equation}
by acting $\varphi_{d_2}(g)$ for a suitable $g\in{\rm GL}_2(k)$ on $M$.
Indeed, such $\varphi_{d_2}(g)$ is implemented by $g=\pmat{a&0\\0&d}$ satisfying 
\begin{center}
$a^{q^2(q+1)}d^{q(q+1)}={M^*_{1,2}}^{-1}$ and $a^{q^2-1}d^{(q^2+1)(q+1)}={M^*_{2,4}}^{-1}$.
\end{center}
% as is the case where $q\ge3$. 
Let $M,N$ satisfy \eqref{q2f}.
Assume that
%\begin{center}
$k^{\times}{M}^{\varphi_{d_2}}=k^{\times}{N}^{\varphi_{d_2}}$ for $M,N\in{\mathcal{Q}_{d_2}}'$.
%\end{center}
%where $M,N$ satisfy the condition \eqref{q2f}.
Then
%by the same argument with the proof for $\mathcal{Q}_{d_1}$, 
there are $\lambda\in k^\times$ and $h\in{\rm GL}_2(k)$ such that
$$
\lambda\left\{\lt\varphi_{d_2}(h)\right\}^*{M}^*\left\{{\varphi_{d_2}(h)}^{(q)}\right\}^*={N}^*.
$$
%From this equality, 
%It immediately follows that
% shows that
%if , 
This equality implies that there are $\alpha,\delta\in k$ such that
%implies that
%$h=\pmat{\alpha&0\\0&\delta}$ with
$$
(\alpha^{-1}\delta)^9=1,\ \ \ \ (\alpha^{-1}\delta)^6M^*_{1,4}=N^*_{1,4}.
$$
%These equalities mean 
It immediately follows from this that $\xi\,M^*_{1,4}=N^*_{1,4}$ for some $\xi\in{\mathbb{F}_4}^\times$.
% a cube root of unity. 
Consequently, if $\xi\,M^*_{1,4}\neq N^*_{1,4}$ then $k^{\times}{M}^{\varphi_{d_2}}\neq k^{\times}{N}^{\varphi_{d_2}}$.
% belong to distinct classes of $k^{\times}\backslash \mathcal{Q}_{d_2}/\sim$
Since
%pairs of 
such $(M^*_{1,4},N^*_{1,4})\in k^\times\times k^\times$ exists infinitely many, one has $|k^{\times}\backslash \mathcal{Q}_{d_2}/\sim|=\infty$, and so the proof has been completed.
\end{proof}
Theorem \ref{thm1} immediately follows from Proposition \ref{prop31} and Lemmas \ref{bico},\ref{lem3}.
% the proof has been completed.
%one sees immediately that $\mathcal{T}_{q+1}$,$\mathcal{T}_{q(q+1)}$ for $q=2$ split into infinitely many ${\rm Aut}(X_A)$-orbits and ${\rm Aut}(X_A)$ acts transitively on $\mathcal{T}_{q+1},\mathcal{T}_{q(q+1)},\mathcal{T}_{d_3}$ for $q\ge3$. From the above, a proof of Theorem \ref{thm1} has been completed.
\begin{rem}
For $q\ge3$, the last statment of Theorem \ref{thm1} is clear from Lemma \ref{lem3} and Remark \ref{remRl}. 
%that the moduli space of ${\rm Aut}(X_A)\backslash \mathcal{T}_{d_1}$ is $\left\{(0,0)\right\}\subset\mathbb{A}^2$ and those of ${\rm Aut}(X_A)\backslash \mathcal{T}_{d_i}$ for $i=2,3$ are $\left\{0\right\}\subset\mathbb{A}^1$.
\end{rem}

\begin{proof}[Proof of Theorem \ref{modu}]%\label{rem1}
%As is seen in the above proof, when $q=2$ the set $k^{\times}\backslash \mathcal{Q}_i/\sim$ and therefore the set $\mathcal{T}_{d_i}$ modulo ${\rm Aut}(X_A)$ for $i=1,2$ are one-dimensional over $k$, where $d_1=3$ and $d_2=6$. 
%As is immediately seen from
\mbox{}\\
{\bf Case of $
%{\rm Aut}(X_A)\backslash 
\mathcal{T}_{d_2}$:} From the proof of Lemma \ref{lem3} for 
%$k^{\times}\backslash \mathcal{Q}_{d_2}/\sim$ with 
$q=2$, it immediately follows that $k^{\times}\backslash \mathcal{Q}_{d_2}/\sim$ is bijectively parametrized by $k/{\mathbb{F}_{4}}^\times\simeq\mathbb{A}^1$ and thus so is ${\rm Aut}(X_A)\backslash \mathcal{T}_{d_2}$
%Therefore 
by Lemma \ref{bico} and Proposition \ref{prop31}.
%${\rm Aut}(X_A)\backslash \mathcal{T}_{d_2}$ is bijectively parametrized by the set
%has the unique coset $k^{\times}{R_2}^{\varphi_{d_2}}$ such that $M^*_{1,4}=M^*_{4,2}=0$ for every $M=(M_{i,j})\in k^{\times}{R_2}^{\varphi_{d_2}}$.
%there exists the unique element $k^\times {M}^{\varphi_{d_2}}$ of $k^{\times}\backslash \mathcal{Q}_{d_2}/\sim$ such that $M^*_{1,4}=M^*_{4,2}=0$.
%Therefore $k^{\times}\backslash \mathcal{Q}_{d_2}/\sim$, and hence ${\rm Aut}(X_A)\backslash \mathcal{T}_{q(q+1)}$
%is in one-to-one correspondence with $\left(k^\times/{\mathbb{F}_{4}}^\times\right)\cup\left\{0\right\}
%\langle\xi\rangle
%$.
%, where $0$ corresponds to the unique orbit for $M$ with $M^*_{1,4}=M^*_{4,2}=0$.
%, where $\langle\xi\rangle=\{1,\xi,\xi^2\}$. 
%This means that ${\rm Aut}(X_A)\backslash \mathcal{T}_{q(q+1)}$ for $q=2$ is in one-to-one correspondence with $k^\times/{\mathbb{F}_{4}}^\times\sqcup\left\{0\right\}$.
\\
{\bf Case of $
%{\rm Aut}(X_A)\backslash 
\mathcal{T}_{d_1}$:}
%For $\mathcal{T}_{q+1}$ with $q=2$, 
%Similarly we will evaluate the size of ${\rm Aut}(X_A)\backslash \mathcal{T}_{d_1}$ for $q=2$, 
Let $M$ be a representative of every $k^{\times}M^{\varphi_{d_l}}\in k^{\times}\backslash \mathcal{Q}_{d_1}/\sim$ such that $M\in{\mathcal{Q}_{d_1}}'$.
% so that $M\in{\mathcal{Q}_{d_1}}'$.
%as in the proof of Lemma \ref{lem3}, and then $M^*\in\mathcal{B}_{d_1}$. 
Since $M^*\in\mathcal{B}_{d_1}$, we denote $M^*$ by ${M_B}^*$, where
$
B=\pmat{b_{11}&b_{12}&b_{13}\\b_{21}&b_{22}&b_{23}}
$
that is the submatrix formed by rows $1$-$2$ and columns $2$-$4$ of $M^*$ and is the matrix determining $M^*$. 
We prepare the following lemma which is
% by applying 
a variant of %Lemma 5 in
\cite[Lemma 5]{O}. 
\begin{lem}\label{lem36}
Let $q=2$.
% and ${M_B}^*\in\mathcal{B}_{d_1}$ for
%$
%B=\pmat{b_{11}&b_{12}&b_{13}\\b_{21}&b_{22}&b_{23}}
%$
%that is the submatrix formed by rows $1$-$2$ and columns $2$-$4$ of ${M_B}^*$.
%a submatrix of $M_B$ characterizing that.
%and then denote by 
%the case I of Lemma \ref{mat} and
%Suppose $q=2$. 
For all $g=\pmat{a&b\\c&d}\in{\rm GL}_2(k)$
%by direct computation 
we have
% and $B\in \mathcal{M}_{2,3}$, by a direct computation 
$$
\left\{\lt\varphi_{d_1}(g)\right\}^*{M_B}^*\left\{\varphi_{d_1}(g)^{(2)}\right\}^*=\det(g)^2{M_{\lt gBg'^{(2)}}}^*,
$$
where $g'=\pmat{a^2&0&b^2\\ac&{\rm det}(g)&bd\\c^2&0&d^2}$.
% for $g=\pmat{a&b\\c&d}$.
\end{lem}
\begin{proof}
Due to straightforward calculation.
\end{proof}
Put $B'=\pmat{b_{11}&b_{13}\\b_{21}&b_{23}}$ that is the submatrix formed by columns $1$ and $3$ of $B$. Since ${\rm rank}(B')\neq0$
%is greater than zero 
by definition, the proof can be divided into two cases.
% presented in the above lemma.
%Since $B'\in{\rm GL}_2(k)$ by definition, 
%Suppose that 
\\
%{\bf(Case $B'\not\in{\rm GL}_2(k)$)} 
{\bf(The case where $
%B'\in{\rm GL}_2(k)
{\rm rank}(B')=1
$)}\\
In the case, there is $\mu\in k^\times$ such that $(b_{11},b_{21})=\mu(b_{13},b_{23})\neq(0,0)$. Take $\pmat{1&0\\0&\lambda}
%\in{\rm GL}_2(k)
$ with
%$\lambda\in k^\times$ with 
$\lambda^4=\mu$ as $g$ in Lemma \ref{lem36}. Then
% the entries of 
$$
\det(g)^2\,\lt gBg'^{(2)}=\pmat{\lambda^2 b_{11}&\lambda^4b_{12}&\lambda^6b_{13}\\
\lambda^3b_{21}&\lambda^5b_{22}&\lambda^7b_{23}}
$$
%$ satisfy
% by Lemma \ref{lem36},
with $(\lambda^2 b_{11},\lambda^3b_{21})=(\lambda^6b_{13},\lambda^7b_{23})$,
% and $(\lambda^4b_{12},\lambda^5b_{22})$, 
that is, $\det(g)^2\,\lt gBg'^{(2)}$ can be expressed by the form $\pmat{x&z&x\\y&w&y}$.
%Further, take $$\pmat{1&\\0&\lambda}$
Further, taking $h=\pmat{0&1\\1&0}$ as $g$ in the lemma,
% and set it to $h$. 
%Then 
this rearranges the entries of $\pmat{x&z&x\\y&w&y}$ symmetrically about the center.
% and it  acts on $\mathbb{A}^2$ as swap the order of the coordinates. 
Now, one may assume $xw-yz=1$ by multiplying a suitable element of $k^\times$ to $B$ since $xw-yz={\rm det}\left(\pmat{x&z\\y&w}\right)\neq0$ since ${\rm rank}(B)=2$. Let $\langle h\rangle$ denote the subgroup of ${\rm SL}_2(k)$ generated by $h$. The group $\langle h\rangle$ acts freely on ${\rm SL}_2(k)$ by left multiplication. 
%It immediately follows from the above argument that 
Therefore the subset $U$ of $k^{\times}\backslash \mathcal{Q}_{d_1}/\sim$ defined by ${\rm rank}(B')=1$
% ${\rm det}(B')=0$
is bijectively parameterized by $\langle h\rangle\backslash{\rm SL}_2(k)$. 
%which is a smooth rational
%$3$-dimensional affine 
%hypersurface of
%of degree $3$ in 
%$\mathbb{A}^4$.
% of $\mathbb{A}^4$. 
The coordinate ring of ${\rm SL}_2(k)$ is defined by $xw+yz+1=0$ and that of
$\langle h\rangle\backslash{\rm SL}_2(k)$ equals to the 
ring of the invariants of ${\rm SL}_2(k)$ under $\langle h\rangle$ which is generated by $X=x+y,Y=z+w,Z=xy,W=zw$. Therefore $\langle h\rangle\backslash{\rm SL}_2(k)$ is an affine hypersurface defined by $X^2W+Y^2Z+XY+1=0$ and this is clearly smooth and rational.
\\
{\bf(The case where $
%B'\in{\rm GL}_2(k)
{\rm rank}(B')=2
$)}\\
In the case, $B'\in{\rm GL}_2(k)$, and so there is $B''=\pmat{\alpha&\beta\\\gamma&\delta}\in{\rm GL}_2(k)$ such that $B'=\lt B''{B''}^{(4)}$ by replacing $q$ with $q^2$ in Proposition \ref{prop1}.
%It is easy to see that
Then
$$
B=\pmat{\alpha&b_{12}&\gamma\\\beta&b_{22}&\delta}\pmat{\alpha^4&0&\beta^4\\0&1&0\\\gamma^4&0&\delta^4}.
$$
%For a given $B'$,
Multiplying $\lambda\in k^\times$ with $\lambda^{10}={\rm det}(B')$ to $B''$, one may assume that ${\rm det}(B'')=1$.
Applying the case where $g={B''}^{-1}$ of Lemma \ref{lem36},
%, by taking ${\lt B'}^{-1}$ as $G$ 
we have
$$
\left\{\lt\varphi_{d_1}\left({B''}^{-1}\right)\right\}^*{M_B}^*\left\{\varphi_{d_1}\left({B''}^{-1}\right)^{(2)}\right\}^*={M_{B'''}}^*,
$$
where
$$
B'''=\pmat{1+y^2z&z&x^2z\\y^2w&w&1+x^2w}\ {\rm for}\ x=\alpha\beta,\ y=\gamma\delta,\ \pmat{z\\w}=\pmat{\delta&\gamma\\\beta&\alpha}\pmat{b_{12}\\b_{22}}.
%\begin{aligned}&x=\alpha\beta,&&y=\gamma\delta,\\&z=\delta b_{12}+\gamma b_{22},&&w=\beta b_{12}+\alpha b_{22}\end{aligned}.
$$
%This indicates that the size of ${\rm Aut}(X_A)\backslash \mathcal{T}_{q+1}$ for $q=2$ is bounded above by $k^4$. 
%More precisely, ${\rm Aut}(X_A)\backslash \mathcal{T}_{q+1}$ for $q=2$ corresponds one-to-one with $(k^\times\times k)^2\sqcup\left\{(0,0,0,0)\right\}$. 
%
%Combining the above discussion with action by $h$ described in the previous case,
The group $\langle h \rangle$ acts on $k^2$ as $S_2$ the symmetric group of order $2$.
%, i.e. $(x,y)\leftrightarrow(y,x)$ and $(z,y)\leftrightarrow(y,x)$.
Hence
%it is easy to see from the above discussion that 
the complement $V$ of $U$ in $k^{\times}\backslash \mathcal{Q}_{d_1}/\sim$ is parameterized by $(k^2/S_2)^2\simeq\mathbb{A}^4$.
% where $S_2$ is the symmetric group of order $2$ which acts on $k^2$ as swap the order of the coordinates.
%has a complete set $\mathcal{S}$ of representatives consisting of the matrices $M_{B'''}$ corresponding to the form $B'''$. Thus $\mathcal{S}$ is parametrized by $k^4$. 
Put ${B_1}''',{B_2}'''$ be the matrices with $(x_1,y_1,z_1,w_1),(x_2,y_2,z_2,w_2)$ of the above form, respectively. If ${B_1}'''={B_2}'''$ then $z_1=z_2$ and $w_1=w_2$. Further if $z_1\neq0$ or $w_1\neq0$ then $x_1=x_2$ and $y_1=y_2$.
Therefore the parameterization is bijective except for a half coordinate plane $H/S_2$ in $\mathbb{A}^4$
%and on its complement 
as
%an open toric subvariety of $\mathbb{A}^4$ and  as
% described below,
\begin{align*}
%$$
\mathbb{A}^4-(H/S_2)
%\simeq(k^2\setminus\{(0,0)\})/S_2\simeq\mathbb{A}^2\setminus\{(0,0)\}
\longrightarrow V-\left\{k^{\times}{M_{{B_0}'''}}^{\varphi_{d_1}}\right\}\ {\rm and}\ 
H/S_2\longrightarrow\left\{k^{\times}{M_{{B_0}'''}}^{\varphi_{d_1}}\right\},
%$$
\end{align*}
%$$
%$$
%respectively, 
where
$$
H=\left\{(x,y,0,0)\in\mathbb{A}^4\ \vert\ x,y\in\mathbb{A}^1\right\}\ {\rm and}\ {B_0}'''=\pmat{1&0&0\\0&0&1}.
$$
Since
\begin{align*}
\mathbb{A}^4-(H/S_2)&\simeq(k^2/S_2)^2-\left[(k^2/S_2)\times\left\{(0,0)\right\}\right]\\
&\simeq(k^2-\{(0,0)\})/S_2\simeq\mathbb{A}^2-\{(0,0)\},
\end{align*}
there is a bijective paramerization $\mathbb{A}^2\rightarrow V$ such that $(0,0)\mapsto k^{\times}{M_{{B_0}'''}}^{\varphi_{d_1}}$.
% which is an open subvariety of the second projection $\pi_2:
%\mathbb{A}^4=
%\mathbb{A}^2\times\mathbb{A}^2\rightarrow\mathbb{A}^2$.
%Especially, there is the unique ${\rm Aut}(X_A)$-orbit of $\mathcal{T}_{q+1}$ corresponding to $M_B$ with $b_{12}=b_{22}=0$ by
%$$
%\pmat{b_{12}\\b_{22}}=\bm{0}\iff\pmat{z\\w}=\bm{0}.
%$$
%Similarly for the case where $B'\not\in{\rm GL}_2(k)$ we can show the same assertion by a simple argument.
\end{proof}
To obtain Corollary \ref{cor2}, suppose that $A$ is Hermitian. Then $X_A$ and ${\rm Aut}(X_A)$ are defined over $\mathbb{F}_{q^2}$. If $q=2$ then $X_A$ is rational, and so the corollary is clear. Suppose $q\ge3$. By Theorems \ref{th0} and \ref{thm1}, all curves of $\mathcal{T}_{q+1},\mathcal{T}_{q(q+1)},\mathcal{T}_{\frac{q(q+1)}{2}}$ are projectively isomorphic over $\mathbb{F}_{q^2}$ to the three curves $C_{q+1},C_{q(q+1)},C_{\frac{q(q+1)}{2}}$ whose the coordinates %in $\mathbb{P}^3$ 
%$\lt(x_0,x_1,x_2,x_3)\in\mathbb{P}^3$ 
are given by
%$C_{q+1},C_{q(q+1)},C_{\frac{q(q+1)}{2}}$ defined by
%\begin{align*}
$$
%\mathbb{P}^3\ni
\lt(x_0,x_1,x_2,x_3)=\begin{cases}
%C_I(q+1,1,q)
%C_{q+1}&=\left\{
\lt(s^{q+1},s^qt,st^q,t^{q+1}),\\
%\in\mathbb{P}^3 \middle\vert\ {}^\mathrm t\hspace{-0.5mm}(s,t)\in\mathbb{P}^1\right\},\\
%C_I(q(q+1),q+1,q^2+1)
%C_{q(q+1)}&=\left\{
\lt(s^{q(q+1)},s^{q^2-1}t^{q+1},s^{q-1}t^{q^2+1},t^{q(q+1)}),\\
%\in\mathbb{P}^3\ \middle\vert\ {}^\mathrm t\hspace{-0.5mm}(s,t)\in\mathbb{P}^1\right\},\\
%C_I\left(\frac{q(q+1)}{2},\frac{q+1}{2},\frac{q^2+1}{2}\right)
%C_{\frac{q(q+1)}{2}}&=\left\{
\lt(s^{\frac{q(q+1)}{2}},s^{\frac{q^2-1}{2}}t^{\frac{q+1}{2}},s^{\frac{q-1}{2}}t^{\frac{q^2+1}{2}},t^{\frac{q(q+1)}{2}}),
\end{cases}
%\in\mathbb{P}^3\ \middle\vert\ {}^\mathrm t\hspace{-0.5mm}(s,t)\in\mathbb{P}^1\right\},
%\end{align*}
$$
respectively. Consider the maps
\begin{align*}
%\phi_{q+1}:
&\,
\begin{aligned}
\lt(s,1)\\
\lt(1,0)
\end{aligned}
\overset{\phi_{1}}{\longmapsto}
\begin{aligned}
&\lt(s^{q+1},s^q,s,1)\\
&\lt(1,0,0,0)
\end{aligned}
\overset{\psi_{1}}{\longmapsto}
\begin{aligned}
&\lt(x_0,x_1)=\lt(s,1)\\
&\lt(x_0,x_1)=\lt(1,0),
\end{aligned}
\\
%\phi_{q(q+1)}:
&\,\lt(s,1)\overset{\phi_{2}}{\longmapsto}\lt(s^{q(q+1)},s^{q^2-1},s^{q-1},1)\overset{\psi_{2}}{\longmapsto}\lt(x_0x_1{x_2}^{-1},1)=\lt(s^{2q^2},1),\\
%\phi_{\frac{q(q+1)}{2}}:
&\,\lt(s,1)\overset{\phi_{3}}{\longmapsto}\lt(s^{\frac{q(q+1)}{2}},s^{\frac{q^2-1}{2}},s^{\frac{q-1}{2}},1)\overset{\psi_{3}}{\longmapsto}\lt(x_0x_1{x_2}^{-1},1)=\lt(s^{q^2},1).
\end{align*}
Obviously, $\phi_1\circ\psi_1$ is identical on $C_{q+1}$ and $\phi_2\circ\psi_2$, $\phi_3\circ\psi_3$ are identical on $C_{q(q+1)}(\mathbb{F}_{2q^2})$, $C_{\frac{q(q+1)}{2}}(\mathbb{F}_{q^2})$ without the singular points, respectively. Therefore we have Corollary \ref{cor2}.  
\section{
%Proof of Theorem \ref{th0}
Determination of the number of curves}
%In this section, we prove Theorem \ref{th2} by determaining the numbers $|\mathcal{T}_{q(q+1)}|,|\mathcal{T}_{d_3}|$ of the curves for $q\ge3$. 
%In this case, one has
%Let $q\ge3$ and $i=1,2,3$.
%corresponding to $k^{\times}F_i{{\rm Im}(\varphi_{d_i})}$ %such that
%$\lt F_iA{F_i}^{(q)}\in {\mathcal{Q}_i}'$ and 
%$\lt F_i^*A{F_i^*}^{(q)}={R_i}^*$.
% for $R_i$ of 
%o begin with, 
%\begin{prop}\label{lemstb}
%${\rm Stab}(C_2),{\rm Stab}(C_3)$ are commutative groups of orders $(q+1)(q^3+1),(q+1)(q^3+1)/2$, respectively.
%\end{prop}
\begin{proof}[Proof of Theorem \ref{th2}]
%Since the case of $C_1$ had been proven by \cite{O}, it suffices to show the those of $C_2,C_3$.
%Let $q=2$.
%Let $q\ge3$.
Let $C_{{F_l}^*}=C_{{F_l}^*}(d_l,i_l,j_l)$ be a curve of $\mathcal{T}_{d_l}$ for $l=1,2,3$ and
let ${\rm Stab}(C_{{F_l}^*})$ be the subgroup of ${\rm Aut}(X_A)$ stabilizing $C_{{F_l}^*}$.
%the stabilizer group of 
%with $\varGamma(C_i)\subset{\rm GL}_4(k)$ 
Put ${\rm Stab}(C_{{F_l}^*})=\varGamma(C_{{F_l}^*})/k^{\times}I$ , where $\varGamma(C_{{F_l}^*})\subset{\rm GL}_4(k)$, and put $F_l=(\bm{f}_1,\cdots,\bm{f}_{d_l+1})\in\mathcal{M}_{4,d_l+1}$
% the $4$-by-$(d_i+1)$ matrix having 
so that
$$
\bm{f}_j=\begin{cases}
\bm{f}^*_{f_l(j)}&{\rm for}\ j\in D_{d_l}\\
\bm{0}&{\rm for}\ j\not\in D_{d_l}
\end{cases},
$$
where ${F_l}^*=(\bm{f}^*_1,\bm{f}^*_2,\bm{f}^*_3,\bm{f}^*_4)$.
% where $\varGamma(C_{F^*})\subset{\rm GL}_4(k)$. 
By definition, for every $H_l\in\varGamma(C_{{F_l}^*})$ there is $\lambda_l\in k^{\times}$ such that $\lt H_lA{H_l}^{(q)}=\lambda_l A$ and
$$
H_lk^\times {F_l}{\rm Im}(\varphi_{d_l})=k^\times F_l{\rm Im}(\varphi_{d_l}),
$$ 
and so
$$
H_lk^\times {F_l}^*{\rm Im}(\varphi_{d_l})^*=k^\times {F_l}^*{\rm Im}(\varphi_{d_l})^*.
$$
% with $i=2,3$, 
Hence there is $g_l\in{\rm GL}_2(k)$ such that $H_l={F_l}^*\varphi_{d_l}(g_l)^*{{F_l}^*}^{-1}$. Note that ${\rm Stab}(C_{{F_l}^*})\simeq{\rm Stab}(C_{{G_l}^*})$ for any $C_{{F_l}^*},C_{{G_l}^*}\in \mathcal{T}_{d_l}$. Indeed, the projective transformation $T_l=k^{\times} {{G_l}^*}{{F_l}^*}^{-1}$ mapping $C_{{F_l}^*}$ to $C_{{G_l}^*}$ induces the isomorphism
%$$
%{\rm Stab}(C_{{F_l}^*})\overset{\simeq}{\longrightarrow} T_l\,{\rm Stab}(C_{{F_l}^*})\,{T_l}^{-1}={\rm Stab}(C_{{G_l}^*}).
%$$
\begin{align*}
\begin{array}{cccc}\hspace{-5mm}{\rm Stab}(C_{{F_l}^*})&\overset{\simeq}{\longrightarrow}&T_l\,{\rm Stab}(C_{{F_l}^*})\,{T_l}^{-1}&={\rm Stab}(C_{{G_l}^*})\\
\hspace{-5mm}\rotatebox{90}{$\in$}&&\rotatebox{90}{$\in$}&\\
\hspace{-5mm}H_lk^\times&\longmapsto&T_l\,H_l\,{T_l}^{-1}k^\times.&
\end{array}
\end{align*}
Therefore it suffices to prove the assertion for
%each case of 
the curves corresponding to ${R_1}^*,{R_2}^*,{R_3}^*$ of Remark \ref{remRl}. Since the case of ${R_1}^*$ had been proven in \cite{O}, we will show those of ${R_2}^*,{R_3}^*$.

Let
%$R_i$ be as in Lemma \ref{lem3} and 
$C_{{F_l}^*}$ be the curve of $\mathcal{T}_{d_l}$ with $\lt {F_l}^*A{{F_l}^*}^{(q)}={R_l}^*$ for $l=2,3$. By the above argument, for every $H_l\in\varGamma(C_{{F_l}^*})$ there  are $\lambda\in k^{\times}$ and  $g_l\in{\rm GL}_2(k)$ such that $\lt H_lA{H_l}^{(q)}=\lambda A$ and $H_l={F_l}^*\varphi_{d_l}(g_l)^*{{F_l}^*}^{-1}$. 
%In the equality \eqref{eq1}, 
Since one may assume $\lambda=1$, multiplying $g_l$ by a suitable element of $k^{\times}$ if necessary,
%be the curve corresponding to $k^{\times}F_i{{\rm Im}(\varphi_{d_i})}\in k^{\times}\backslash S_i/{\rm Im}(\varphi_{d_i})$ such that $\lt F_i^*A{F_i^*}^{(q)}={R_i}^*$.
one has 
\begin{equation}\label{eq1}
\lt\left\{\varphi_{d_l}(g_l)^*\right\}{R_l}^*\left\{\varphi_{d_l}(g_l)^*\right\}^{(q)}={R_l}^*.%\ \ {\rm for\ all}\ \lambda\in k^{\times}.
\end{equation}
Let us characterize $g_l\in{\rm GL}_2(k)$ satisfying the above equality.
% \eqref{eq1}. 
\\
{\bf Case of $C_{{F_2}^*}$:} By definition,
%the matrix $\varphi_{d_i}(g)^*$ is of the following form:
\begin{equation*}
\varphi_{d_2}(g_2)^*=
\begin{pmatrix}
a^{q^2+q}&0&0&b^{q^2+q}\\
a^{q^2-1}c^{q+1}&\alpha&\beta&b^{q^2-1}d^{q+1}\\
a^{q-1}c^{q^2+1}&\gamma&\delta&b^{q-1}d^{q^2+1}\\
c^{q^2+q}&0&0&d^{q^2+q}\\
\end{pmatrix}
\ \ 
{\rm for}
\ 
g_2=
\begin{pmatrix}
a&b\\
c&d
\end{pmatrix},
\end{equation*}
where 
%$\alpha,\beta,\gamma,\delta$ are polynomials in $a,b,c,d$ with
\begin{equation*}
\begin{array}{lll}
\alpha=a^{q^2-q-2}\det(g_2)^{q+1},&\beta=a^{q-2}b^{q^2-q}d^{q}\det(g_2),\\
\gamma=0,&\delta=a^{q-2}d^{q^2}\det(g_2).
\end{array}
\end{equation*}
%In the equality \eqref{eq1}, multiplying $g$ by an appropriate element of $k^{\times}$, we may assume that $\lambda=1$. 
Then it follows by straightforward calculation that 
$$
g_2\ {\rm satisfies}\ \eqref{eq1}\ {\rm if\ and\ only\ if}\ b=c=0,\ a=\rho^n\zeta^{mq^2}\ {\rm and}\ d=\zeta^m,
$$
where $n,m\in\mathbb{Z}$
% with $1\le n\le(q^3+1)(q+1)$ 
and $\rho,\zeta$ are primitive $(q+1),(q+1)(q^3+1)$-th roots of unity, respectively. Further, for such $g_2$ one can easily see that
\begin{center}
$H_2={F_2}^*\varphi_{d_2}(g_2)^*{{F_2}^*}^{-1}$ is equal to $I$ if and only if $n,m$ are divisible by $q+1,q^3+1$, respectively.
\end{center}
%$$
%\xi^{nq^2}=\xi^n\ {\rm if\ and\ only\ if}\ n\ {\rm is\ a\ multiple\ of}\ q^3+1.
%$$ 
Hence
% by Homomorphism Theorem,
%by First Isomorphism Theorem 
there is an isomorphism% (by Homomorphism Theorem)
%\begin{align*}
%%\left\{\mathbb{Z}/\{(q+1)\}\mathbb{Z}\right\}\times\left\{\mathbb{Z}/\{(q+1)(q^3+1)\}\mathbb{Z}\right\}/\langle(q+1,q^3+1)\rangle\ni(\overline{n},\overline{m})\\
%\mathbb{Z}/(q+1)\mathbb{Z}\times\mathbb{Z}/(q^3+1)\mathbb{Z}
%%\mathbb{Z}_{q+1}\times\mathbb{Z}_{q^3+1}
%&\longrightarrow{\rm Stab}(C_{{F_2}^*});\\
%(\overline{n},\overline{m})
%&\longmapsto k^{\times}{{F_2}^*} \varphi_{d_2}\left(
%\setlength{\arraycolsep}{1pt}
%\begin{pmatrix}
%\rho^{\overline{n}}\zeta^{\overline{m}q^2}&0\\
%0&\zeta^{\overline{m}}
%\end{pmatrix}
%\right)^*{{F_2}^*}^{-1},
%%\end{multline*}
%\end{align*}
\begin{align*}
\begin{array}{ccc}\hspace{-5mm}\mathbb{Z}/(q+1)\mathbb{Z}\times\mathbb{Z}/(q^3+1)\mathbb{Z}
&\overset{\simeq}{\longrightarrow}&{\rm Stab}(C_{{F_2}^*})\\
\hspace{-5mm}\rotatebox{90}{$\in$}&&\rotatebox{90}{$\in$}\\
(\overline{n},\overline{m})
&\longmapsto&{{F_2}^*} \varphi_{d_2}\left(
\setlength{\arraycolsep}{1pt}
\begin{pmatrix}
\rho^{\overline{n}}\zeta^{\overline{m}q^2}&0\\
0&\zeta^{\overline{m}}
\end{pmatrix}
\right)^*{{F_2}^*}^{-1}k^{\times}.
\end{array}
\end{align*}
%where $\langle(*,*)\rangle$ is the cyclic subgroup generated by an element $(*,*)$.
%of the additive group $\left\{\mathbb{Z}/\{(q+1)\}\mathbb{Z}\right\}\times\left\{\mathbb{Z}/\{(q+1)(q^3+1)\}\mathbb{Z}\right\}$. 
%In particular, the order of ${\rm Stab}(C_2)$ is
%$(q+1)\times(q+1)(q^3+1)/(q+1)=(q+1)(q^3+1)$.\\ 
{\bf Case of $C_{{F_3}^*}$:}
%Although we can show the case by similar argument,
% the case of $i=3$, 
%the situation becomes somewhat complicated. Indeed, 
In this case, 
%the entries of the submatrix $\pmat{\alpha&\beta\\\gamma&\delta}$ formed by rows $2$-$3$ and columns $2$-$3$ of $\varphi_{d_3}(g_3)^*$ for $g_3=\begin{pmatrix}
%a&b\\
%c&d
%\end{pmatrix}
%$
%% $\alpha,\beta,\gamma,\delta$ in $\varphi_{d_3}(g_3)^*$ 
%are as follows:
\begin{equation*}
\varphi_{d_3}(g_3)^*=
\begin{pmatrix}
a^{\frac{q^2+q}{2}}&0&0&b^{\frac{q^2+q}{2}}\\
a^{\frac{q^2-1}{2}}c^{\frac{q+1}{2}}&\alpha&\beta&b^{\frac{q^2-1}{2}}d^{\frac{q+1}{2}}\\
a^{\frac{q-1}{2}}c^{\frac{q^2+1}{2}}&\gamma&\delta&b^{\frac{q-1}{2}}d^{\frac{q^2+1}{2}}\\
c^{\frac{q^2+q}{2}}&0&0&d^{\frac{q^2+q}{2}}\\
\end{pmatrix}
\ \ 
{\rm for}
\ 
g_3=
\begin{pmatrix}
a&b\\
c&d
\end{pmatrix},
\end{equation*}
where 
\begin{align*}
\alpha&=a^{\frac{q^2-q-2}{2}}\sum_{j=0}^{\frac{q+1}{2}}\binom{\frac{q+1}{2}}{j}\binom{\frac{q^2-1}{2}}{j}a^{\frac{q+1-2j}{2}}b^{j}c^{j}d^{\frac{q+1-2j}{2}},\\
\beta&=b^{\frac{q^2-q}{2}}d\sum_{j=0}^{\frac{q-1}{2}}\binom{\frac{q+1}{2}}{j}\binom{\frac{q^2-1}{2}}{\frac{q-1-2j}{2}}a^{\frac{q-1-2j}{2}}b^{j}c^{j}d^{\frac{q-1-2j}{2}},\\
\gamma&=c^{\frac{q^2-q}{2}}d\sum_{j=0}^{\frac{q-1}{2}}\binom{\frac{q-1}{2}}{j}\binom{\frac{q^2+1}{2}}{\frac{q+1-2j}{2}}a^{\frac{q-1-2j}{2}}b^{j}c^{j}d^{\frac{q-1-2j}{2}},\\
\delta&=d^{\frac{q^2-q+2}{2}}\sum_{j=0}^{\frac{q-1}{2}}\binom{\frac{q-1}{2}}{j}\binom{\frac{q^2+1}{2}}{j}a^{\frac{q-1-2j}{2}}b^{j}c^{j}d^{\frac{q-1-2j}{2}}.
\end{align*}
Here, if $g_3$ satisfies \eqref{eq1}, we have the following equations:
\begin{align}
\label{6}a^{\frac{q-1}{2}}\left(a^{\frac{q^2+1}{2}}\alpha^q-c^{\frac{q^2+1}{2}}\gamma^q\right)&=1,\\
\label{7}d^{\frac{q^3+q}{2}}\left(d^{\frac{q^2-q}{2}}\alpha-b^{\frac{q^2-q}{2}}\gamma\right)&=1,\\
\label{8}a^{\frac{q-1}{2}}\left(a^{\frac{q^2+1}{2}}\beta^q-c^{\frac{q^2+1}{2}}\delta^q\right)&=0,\\
\label{9}d^{\frac{q^3+q}{2}}\left(d^{\frac{q^2-q}{2}}\beta-b^{\frac{q^2-q}{2}}\delta\right)&=0,\\
\label{10}b^{\frac{q-1}{2}}\left(b^{\frac{q^2+1}{2}}\beta^q-d^{\frac{q^2+1}{2}}\delta^q\right)&=0,\\
\label{11}c^{\frac{q^3+q}{2}}\left(c^{\frac{q^2-q}{2}}\beta-a^{\frac{q^2-q}{2}}\delta\right)&=0,\\
\label{12}\delta^{q+1}&=1,\\
\label{13}\gamma&=0.
\end{align}
By \eqref{6}, \eqref{7}, one has $ad\neq0$. Now, assume $bc\neq0$. Then, from \eqref{8},\eqref{9},\eqref{10}\eqref{11} we have a pair of simultaneous equations in $(\beta^q,\delta^q)$ and $(\beta,\delta)$:
$$
\pmat{a^{\frac{q^2+1}{2}}&-c^{\frac{q^2+1}{2}}\\b^{\frac{q^2+1}{2}}&-d^{\frac{q^2+1}{2}}}\pmat{\beta^q\\\delta^q}=\pmat{0\\0},\ \ \pmat{d^{\frac{q^2-q}{2}}&-b^{\frac{q^2-q}{2}}\\c^{\frac{q^2-q}{2}}&-a^{\frac{q^2-q}{2}}}\pmat{\beta\\\delta}=\pmat{0\\0}.
$$
Since $(\beta,\delta)\neq(0,0)$ by \eqref{eq1}, the determinants of the coefficient matrices must be both zero, and so we have the following simultaneous equation in $ad,bc$:
$$
\sum_{j=0}^{\frac{q^2-1}{2}}(ad)^{\frac{q^2-1-2j}{2}}(bc)^j=0,\ \ \sum_{j=0}^{\frac{q^2-q-2}{2}}(ad)^{\frac{q^2-q-2-2j}{2}}(bc)^j=0.
$$
This means that the two polynomials
$$
F_1(x,y)=\sum_{j=0}^{\frac{q^2-1}{2}}x^{\frac{q^2-1-2j}{2}}y^j,\ \ F_2(x,y)=\sum_{j=0}^{\frac{q^2-q-2}{2}}x^{\frac{q^2-q-2-2j}{2}}y^j
$$
have common roots. But it is impossible,
% $F_1$ and $F_2$ have no common roots
because the resultant is equal to $1\neq0$, a contradiction. Hence $bc=0$ must be true, and then using the equations from \eqref{6} to \eqref{13}, one has $b=c=0$, $a=\eta^n\omega^{mq^2}$ and $d=\omega^m$, where $n,m\in\mathbb{Z}$ and
%$n\in\mathbb{Z}$ with $1\le n\le(q^3+1)(q+1)/2$ and 
$\eta,\omega$ are primitive $(q+1)/2$, $(q+1)(q^3+1)/2$-th roots of unity, respectively. Conversely if $g_3$ has these values, it satisfies \eqref{eq1}, and thus we have
%Especially, $\gamma=c^{\frac{q^2-q}{2}}d\,\Phi\not\equiv0$.
%%, more precisely
%%\begin{center}
%%$\gamma=c^{\frac{q^2-q}{2}}d\ \times$ a certain polynomial $\Phi$ of degree $\frac{q-1}{2}$ in $ad,bc$.
%%%\left(\binom{\frac{q^2+1}{2}}{\frac{q^2-q}{2}}(ad)^{\frac{q-1}{2}}+\binom{\frac{q-1}{2}}{1}\binom{\frac{q^2+1}{2}}{\frac{q^2-q+2}{2}}(ad)^{\frac{q-3}{2}}bc+\dots+\binom{\frac{q^2+1}{2}}{\frac{q^2-1}{2}}(bc)^{\frac{q-1}{2}}\right)
%%\end{center}
%Since we do not know what $\Phi=0$ deduces, we only have a weekly condition:
\begin{center}
$g_3$ satisfies \eqref{eq1} if and only if $b=c=0$, $a=\eta^n\omega^{mq^2}$ and $d=\omega^m$.
\end{center}
%where 
%%$n\in\mathbb{Z}$ with $1\le n\le(q^3+1)(q+1)/2$ and 
%$\eta,\omega$ are respectively primitive $(q+1)/2$, $(q+1)(q^3+1)/2$-th roots of unity, and moreover, for $q=3,5$ we can verify by computation that ``and only if'' holds true. 
For such $g_3$ one can easily verify that
\begin{center}
$H_3={F_3}^*\varphi_{d_3}(g_3)^*{{F_3}^*}^{-1}$ is equal to $I$ if and only if $n,m$ are divisible by $(q+1)/2,q^3+1$ respectively.
\end{center}
% by whether $q$ modulo $4$ is $1$ or $3$, respectively. 
Hence there is an isomorphism
%\begin{multline*}
%\left\{\mathbb{Z}/\{(q+1)/2\}\mathbb{Z}\right\}\times\left\{\mathbb{Z}/\{(q+1)(q^3+1)/2\}\mathbb{Z}\right\}/\langle((q+1)/2,q^3+1)\rangle\\\longrightarrow{\rm Stab}(C_3).
$$
\mathbb{Z}/\left((q+1)/2\right)\mathbb{Z}\times\mathbb{Z}/\left(q^3+1\right)\mathbb{Z}\longrightarrow{\rm Stab}(C_{{F_3}^*}),
$$
%\end{multline*}
%and so the order of ${\rm Stab}(C_3)$ is 
%$$
%2^{-1}(q+1)\times2^{-1}(q+1)(q^3+1)/\{2^{-1}(q+1)\}=2^{-1}(q+1)(q^3+1).
%$$
%if $q=3,5$ this is isomorphic, 
and thus the proof of Theorem \ref{th2} has been completed.
\end{proof}
%Since ${\rm Aut}(X_A)$ acts transitively on $\mathcal{T}_{d_i}$ with $q\ge3$ by Lemma \ref{lem3}, 
% for $i=2,3$and so
As is well-known,
\begin{equation*}
|{\rm Aut}(X_A)|=q^6(q^4-1)(q^3+1)(q^2-1).
\end{equation*}
Since
\begin{equation*}
|\mathcal{O}_{d_l}|=|{\rm Aut}(X_A)|/|{\rm Stab}(C_{{F_l}^*})|\ {\rm for}\ l=1,2,3,
%\ \ {\rm for}\ i=2,3.
\end{equation*}
%We know that
%${\rm Aut}(X_A)\simeq{\rm PGU}_4(\mathbb{F}_{q^2})$ (cf. \cite{O}), and so
these numbers
% $|\mathcal{O}_{d_l}|$
%presented in Corollary \ref{mainthm} 
are easily calculated by Theorem \ref{th2}.
 %that $|\mathcal{T}_{q(q+1)}|=q^6(q^4-1)(q-1)$ and $|\mathcal{T}_{d_3}|=2q^6(q^4-1)(q-1)$. Thus

\begin{rem}
Any two ${\rm Aut}(X_A)$-orbits of $\mathcal{T}_{d_l}$ for $l=1,2,3$ are in one-to-one correspondence by projective transformation since any two curves of $\mathcal{T}_{d_l}$ are projectively equivalent as described in the above proof.
% as seen in the above proof.
% of Theorem \ref{th2}.
This particularly implies that each of the combinatorial structures constructed in \cite{O1} which are obtained from the ${\rm Aut}(X_A)$-orbits of $\mathcal{T}_{q+1}$ is unique regardless of the choice of the orbit,
%cf.\ \cite{O1},
although the orbits exist infinitely many when $q=2$ as we have shown.
\end{rem}
%\section*{Declaration}
%The author has no relevant financial or non-financial interests to disclose.

\end{document}